\documentclass[11pt,oneside]{article}

\usepackage[square, numbers]{natbib}

\usepackage{xcolor} 
\usepackage{graphicx} 
\usepackage{amsmath}
\usepackage{amsfonts}
\usepackage{amssymb}
\usepackage{amsthm}
\usepackage{dsfont}
\usepackage{mathrsfs}
\usepackage{bbm}
\usepackage{enumerate}
\usepackage{mathtools}		
\usepackage{hyperref}
\usepackage{bm}

\allowdisplaybreaks

\usepackage[textwidth=2.8cm]{todonotes}

\usepackage[nottoc]{tocbibind}
\usepackage{tocloft}
\renewcommand{\cftsubsubsecpagefont}

\numberwithin{equation}{section}

\newtheorem{thm}{Theorem}[section]
\newtheorem{lem}[thm]{Lemma}
\newtheorem{prop}[thm]{Proposition}

\theoremstyle{definition}
\newtheorem{defn}[thm]{Definition}

\theoremstyle{definition}
\newtheorem{assumption}[thm]{Assumption}

\theoremstyle{remark}
\newtheorem{rem}[thm]{Remark}

\newcommand{\mut}{\tilde{\mu}}
\newcommand{\Gt}{\tilde{G}}
\newcommand{\R}{\mathbb{R}}
\newcommand{\E}{\mathbb{E}}
\newcommand{\N}{\mathbb{N}}

\renewcommand{\P}{\mathbb{P}}

\newcommand{\Bc}{\mathcal{B}}
\newcommand{\Dc}{\mathcal{D}}
\newcommand{\Pc}{\mathcal{P}}
\newcommand{\Cc}{\mathcal{C}}
\newcommand{\Wc}{\mathcal{W}}
\newcommand{\Gc}{\mathcal{G}}
\newcommand{\Mc}{\mathcal{M}}
\newcommand{\Wop}{\mathcal{WOP}}

\newcommand{\Q}{\mathbb{Q}}

\newcommand{\Xt}{\Tilde{X}}

\newcommand{\Lc}{\mathcal{L}}

\newcommand{\mub}{\bar{\mu}}

\newcommand{\eps}{\varepsilon}
\newcommand{\W}{\mathcal{W}}
\newcommand{\x}{\times}

\title{Graphon Particle Systems with Common Noise}
\author{Erhan Bayraktar
        \footnote{Department of Mathematics, University of Michigan. erhan@umich.edu}
        \and
        Xihao He
        \footnote{Department of Mathematics, University of Michigan. hexihao@umich.edu}
        \and
        Donghan Kim
        \footnote{Department of Mathematical Sciences, KAIST. kimdonghan@kaist.ac.kr}
        }

\date{\today}

\begin{document}

\maketitle

\begin{abstract}
    \noindent We study a nonlinear graphon particle system driven by both idiosyncratic and common noise, where interactions are governed by a graphon and represented as positive finite measures. Each particle evolves via a McKean–Vlasov-type SDE with graphon-weighted conditional laws. We prove a law of large numbers for the empirical and interaction measures, using generalized Wasserstein metrics and weak convergence techniques suited for the non-Markovian structure induced by common noise.
\end{abstract}

\tableofcontents
\section{Introduction}

We study a class of interacting particle systems indexed by a continuum of agents, whose interactions are governed by a graphon, a symmetric measurable function $G : [0, 1] \times [0, 1] \to [0, 1]$ that generalizes adjacency matrices to the continuum limit. Specifically, we consider a nonlinear graphon particle system driven by both idiosyncratic and common noise. Each agent evolves according to a McKean-Vlasov-type stochastic differential equation, where the drift and diffusion coefficients depend on the agent’s current state and a graphon-weighted average of the distributions of other agents.

The study of graphon particle systems has gained considerable attention in recent years, motivated by the need to analyze interacting systems with heterogeneous network structures beyond the scope of classical mean-field models \cite{BCW, Bayraktar_Kim_2024, BW2, BW, BWZ, bayraktar2024nonparametricestimatesgraphonmeanfield, Coppini:Graphon, nonlinear:graphon}. Graphon-based frameworks provide a natural continuum limit of graph sequences and have been widely applied not only in mean field games \cite{Aurell2022, graphon:game:GMFG, Graphon:game, Carmona:graphon_games, toymodel:Erdos-Renyi, GCH, Gao:Tchuendom:Caines, Lacker:graphon_games, TCH2, TCH, VMV}, but also in optimization \cite{chen2025graphonparticlesystemsii}, economics \cite{graphon:econometrica}, and epidemiology \cite{graphon:epidem}.

A substantial line of work has focused on approximating graphon particle systems by finite particle systems that retain a similar interaction structure through discretized graphons. A law of large numbers–type convergence between the two systems was first established in \cite{BCW}. Uniform-in-time convergence was later studied in \cite{BW2}, while concentration bounds on the Wasserstein distances between the empirical measure of the finite system and the averaged law of the graphon system were obtained in \cite{Bayraktar_Kim_2024, BW}. Law of large numbers and propagation of chaos results for nonlinear graphon systems, in which the interaction is given as a nonlinear function of local empirical measures, are presented in \cite{nonlinear:graphon}.

The main results of this paper establish a law of large numbers for graphon-weighted empirical measures in the presence of common noise. We consider a sequence of finite particle systems interacting through discretized graphons and show that, as the number of particles tends to infinity, the empirical and interaction measures converge to their continuum counterparts conditioned on the common noise, defined by the graphon system. This convergence is rigorously formulated and quantified using a generalized Wasserstein metric suited for positive finite measures. Our results include the convergence of both the empirical probability measures of the finite system and the graphon-weighted interaction measures that appear in the drift and diffusion terms of the dynamics. We further derive explicit convergence rates for the expected distance between the state in the finite particle system and the corresponding state in the mean-field system.

While existing studies on graphon particle systems typically assume independent Brownian noise for each agent, our work is, to our knowledge, the first to incorporate common noise to prove such convergence results. This addition reflects the realistic features of many complex systems, such as financial markets subject to macroeconomic shocks or biological populations influenced by environmental fluctuations. Consequently, adding common noise to control problems and mean-field games in diverse interaction systems has gained popularity recently \cite{Graphon:Qnoise, Tanpi:Zhou, Xu:Gou}. The presence of common noise introduces substantial analytical difficulties, as the relevant distributions become random conditional laws that evolve in a nonlinear and non-Markovian manner within infinite-dimensional path spaces.

To overcome these challenges, our technical approach is to analyze the joint distributions of the empirical and interaction measures as elements in infinite-dimensional spaces such as $\Pc(\Pc(\Cc^d) \x \Pc(\Cc^d))$, the space of probability measures over pairs of probability measures on the $d$-dimensional continuous function space. We establish the tightness of these joint laws, derive weak convergence using conditional independence structures, and invoke a classical characterization of Wasserstein convergence to conclude convergence in the $\W_2$-metric. This argument enables a rigorous derivation of novel law of large numbers for empirical laws of processes on $[0, T]$ in the presence of common noise (see Theorems \ref{thm: relaxed_bar mu_n to bar mu}, \ref{thm: relaxed_weighted mean W2 continuity}).

Another feature of our work is the treatment of graphon-weighted interactions, which gives rise to positive finite measures rather than probability measures. The interaction term for each agent is represented by a graphon-weighted average of the conditional laws of the population, and such measures generally do not have unit mass. To analyze convergence in this setting, we adopt a generalized Wasserstein-type distance $\Wop_2$, introduced in \cite{extendingWmetric}, which extends the standard 2-Wasserstein distance to the space of positive finite measures. This framework allows us to compare interaction measures in a metrically consistent way and to quantify convergence without requiring normalization, which is essential in the graphon setting where local interaction intensity varies across agents.

Our work builds upon and significantly extends the nonlinear graphon particle system studied in \cite{nonlinear:graphon}, which considered similar graphon-weighted interactions but without common noise. In that setting, the interaction measures were normalized to have unit mass, allowing the use of standard Wasserstein distances on probability measures. By contrast, the inclusion of common noise in our model requires treating the empirical and interaction measures as random (conditional) laws, which are no longer deterministic nor normalized. This motivates our use of the generalized Wasserstein distance $\Wop_2$ for positive finite measures, and also enables us to remove certain conditions imposed on the graphon in \cite{nonlinear:graphon}, for example, $\Vert G(u, \cdot)\Vert_1 := \int_0^1 G(u, v)dv > 0$ and $\int_0^1 \Vert G(u, \cdot)\Vert_1^{-1}du < \infty$. Moreover, we prove the convergence of the finite particle system to the graphon system at the level of process laws. This result strictly strengthens Theorem 4.1 in \cite{nonlinear:graphon}, which only controls time–integrated Wasserstein distances of empirical marginal measures.

The techniques developed in this paper can also be extended to systems with delayed or feedback-type interactions, and may serve as a foundation for studying central limit theorems or control problems in graphon-based mean field models with common noise.

\bigskip

\section{Setting}

\subsection{Graphon particle system and approximating finite particle system}

For a fixed $T > 0$, we study a graphon particle system with common noise:
\begin{align}   \label{eq : graphon particle system with common noise}
	dX_u(t) &= \phi_p \bigg(X_u(t), ~ \int_0^1 G_{u, v} \mu_v(t) dv \bigg) dt + \phi_b \bigg(X_u(t), ~ \int_0^1 G_{u, v} \mu_v(t) dv \bigg) dB_u(t)
	\\
	& \qquad \qquad \qquad + \phi_w \bigg(X_u(t), ~ \int_0^1 G_{u, v} \mu_v(t) dv \bigg) dW(t), \qquad \forall \, u \in [0, 1], \qquad t \in (0, T].    \nonumber
\end{align}
Here, $\phi_p$, $\phi_b$, and $\phi_w$ are $\R^d$, $\R^{d \times d}$, and $\R^{d \times m}$-valued measurable maps on $\mathbb{R}^d \times \mathcal{M}(\mathbb{R}^d)$, with regularity assumptions stated in Assumption~\ref{ass:Lipschitz}, representing the drift, idiosyncratic diffusion, and common-noise diffusion coefficients, respectively. $\{B_u\}_{u \in [0, 1]}$ are independent $d$-dimensional Brownian motions, also independent of another $m$-dimensional Brownian motion $W$, modeling a common noise. We assume that the initial distributions $\{X_u(0)\}_{u \in [0, 1]}$ are independent $\R^d$-valued random variables, independent of $\{B_u\}_{u \in [0, 1]}$ and $W$. The measurable mapping $[0, 1] \times [0, 1] \ni (u, v) \mapsto G_{u, v} \in [0, 1]$, satisfying the symmetry $G_{u, v} = G_{v, u}$ for each $u, v \in [0, 1]$, is called a graphon, and $\mu_u(t)$ is the conditional probability distribution of $X_u(t)$ given $W$ for every $u \in [0, 1]$ and $t \in [0, T]$. Since each $\mu_u(t)$ is a probability measure, the integral $\int_0^1 G_{u, v} \mu_v(t) dv$ should be understood as a positive measure in the sense of
\begin{align*}
	\bigg(\int_0^1 G_{u, v} \mu_v(t) dv \bigg)(A) &:= \int_0^1 G_{u, v} \big(\mu_v(t)(A)\big) dv
\end{align*}
for any Borel-measurable set $A \in \mathcal{B}(\R^d)$. The functions $\phi_{\alpha}$ are defined on $\R^d \times \mathcal{M}(\R^d)$ for $\alpha = p, b, w$, where we denote $\mathcal{M}(\R^d)$ the collection of positive finite measures on $\R^d$.

The system \eqref{eq : graphon particle system with common noise} can be approximated by the finite particle system with $n$ particles:
\begin{align}   
	dX^n_i(t) = \phi_p \bigg(X^n_i(t),  ~\frac{1}{n} \sum_{j=1}^n G^n_{i, j} \delta_{X^n_j(t)} \bigg) dt 
	& + \phi_b \bigg(X^n_i(t),  ~\frac{1}{n} \sum_{j=1}^n G^n_{i, j} \delta_{X^n_j(t)} \bigg) dB_{{i/n}}(t)     \label{eq : finite particle system with common noise}
	\\
	& + \phi_w \bigg(X^n_i(t),  ~\frac{1}{n} \sum_{j=1}^n G^n_{i, j} \delta_{X^n_j(t)} \bigg) dW(t).        \nonumber
\end{align}
Here, we assume $X^n_i(0) = X_{{i/n}}(0)$, and consider for each $n \in \mathbb{N}$ a discretized graphon $G^n :[n] \times [n] \ni (i, j) \mapsto G^n_{i, j} \in [0, 1]$, which approximates the graphon $G$, with the notion $[n] := \{1, 2, \cdots, n\}$ (see Assumption \ref{ass:G^n_and_G} below). 
We can also write $G^n : [0,1] \x [0,1] \longrightarrow [0,1]$ by $(G^n)_{u,v} = G_{\lceil nu \rceil/n,\lceil nv \rceil/n}$, for $u,v \in [0,1]$.
One such example is
\begin{equation}    \label{ex: discrete graphon}
	G^n_{i, j} = \int_0^1 \int_0^1 G_{\frac{i + u}{n}, \frac{j + v}{n}} \, du \, dv, \qquad 
	\forall \, (i, j) \in [n] \times [n].
\end{equation}

\subsection{Extending Wasserstein metric to positive measures}   \label{subsec: Wop}

Since we need to handle positive measures like $\int_0^1 G_{u, v}\mu_{v}(t)dv$ in the graphon particle system \eqref{eq : graphon particle system with common noise}, we introduce in this subsection a new metric on $\mathcal{M}(\R^d)$, i.e. the space of positive finite measures. Extending 2-Wasserstein metric to $\mathcal{M}_2(\R^d)$, which consists of elements in $\mathcal{M}(\R^d)$ with finite second moment, has been studied in \cite{extendingWmetric}, so we first define $\Wop_2$ on $\mathcal{M}_2(\R^d)$ and its properties without proof. In what follows, we shall use the notations for a measurable space $X$:
\begin{itemize}
	\item $m_{\mu}$ is the total mass of a measure $\mu$;
	\item $\mu^{\circ}  = \mu/m_{\mu}$ if $m_{\mu} > 0$ and $\delta_{x_0}$ otherwise, where $\delta_{x_0}$ denotes the Dirac measure at $x_0$;
	\item $M_{x_0}(\mu) := \int_X \Vert x-x_0 \Vert^2 d\mu(x)$ for some $x_0 \in X$;
	\item $T_a(x) = a(x-x_0)+x_0$ for some $a > 0$;
	\item $T_{\#}\mu$ is the pushforward of a measure $\mu$ by a measurable mapping $T$.
	\item $\mathcal{M}(X)$ is the collection of positive finite measures on $X$;
	\item $\mathcal{M}_2(X)$ is the collection of positive finite measures on $X$ with finite second order moment;
	\item $\mathcal{M}_{2,K}(X)$ is the sub-collection of $\mathcal{M}_2(X)$ satisfying $M_{x_0}(\mu) \le Km_{\mu}$ for some $K > 0$;
	\item $\mathcal{P}(X)$ is the collection of probability measures on $X$;
	\item $\mathcal{P}_2(X)$ is the collection of probability measures on $X$ with finite second order moment.
\end{itemize}

\begin{defn}    \label{Defn: WOP R^d}
	Fix an arbitrary reference point $x_0 \in \R^d$. For $\mu$ and $\nu$ in $\mathcal{M}_2(\R^d)$, the 2-Wasserstein On Positive measures ($\Wop_2$) metric is defined by
	\begin{equation}    \label{def: WOP}
		\Wop^2_2(\mu, \nu) = (m_{\mu} - m_{\nu})^2 + \Wc^2_2(T_{m_{\mu}}\#\mu^\circ, T_{m_{\nu}}\#\nu^\circ).
	\end{equation}
\end{defn}

\begin{lem} \label{Lem: WOP R^d}
	$\Wop_2$ on $\mathcal{M}_2(\R^d)$ has the following properties.
	\begin{enumerate}[label=(\roman*)]
		\item $\Wop_2$ is a metric on $\mathcal{M}_2(\R^d)$.
		\item $\Wop^2_2(\mu, \nu) = (m_{\mu} - m_{\nu})^2 + (m_{\mu} - m_{\nu})\big( M_{x_0}(\mu) - M_{x_0}(\nu) \big) + m_{\mu}m_{\nu}\Wc^2_2(\mu^\circ, \nu^\circ)$.
		\item For a positive constant $K>0$, the space $(\mathcal{M}_{2,K}(\R^d), \Wop_2)$ is complete.
	\end{enumerate}
\end{lem}

The proofs of the results in Lemma \ref{Lem: WOP R^d} can be found in \cite{extendingWmetric}. In the following, we describe a similar extension of the 2-Wasserstein metric to $\mathcal{M}_2(\mathcal{C}^d)$. 
Here and in what follows, we denote $\mathcal{C}^d = C([0, T], \R^d)$ the space of $\R^d$-valued continuous functions on $[0, T]$ for a fixed $T>0$, equipped with the topology of uniform convergence. For a Polish space $S$, we equip $\mathcal{P}(S)$ with the topology of weak convergence.

\begin{defn}    \label{Def: WOP C^d}
	Fix $x_0 \in \mathcal{C}^d$. For $\mu$ and $\nu$ in $\mathcal{M}_2(\mathcal{C}^d)$, the 2-Wasserstein On Positive measures ($\Wop_2$) metric is defined by
	\begin{equation}    \label{def: WOP C^d}
		\Wop^2_2(\mu, \nu) = (m_{\mu} - m_{\nu})^2 + \Wc^2_2(T_{m_{\mu}}\#\mu^\circ, T_{m_{\nu}}\#\nu^\circ).
	\end{equation}
\end{defn}
In particular, if $\mu, \nu \in \mathcal{P}(\mathcal{C}^d)$, then 
$
\Wop_2(\mu, \nu) = \Wc_2(\mu, \nu),
$
thus $\Wop_2$ is a generalization of $\Wc_2$ to positive finite measures on $\Cc^d$. Checking that $\Wop_2$ is a metric is straightforward; its proof is essentially the same as that of Lemma~\ref{Lem: WOP R^d} (i) (Theorem 2 of \cite{extendingWmetric}), so we omit the proof.

For $x, y \in \mathcal{C}^d$ and $a, b \ge 0$, we have
\begin{align}
	\big\Vert a(x-x_0)&-b(y-x_0) \big\Vert^2_{*, T} = \sup_{t \in [0, T]} \Big\vert a\big(x(t)-x_0(t)\big)-b\big(y(t)-x_0(t)\big) \Big\vert^2  \label{ineq: sup t}
	\\
	& = \sup_{t \in [0, T]} \Big( a(a-b)\big(x(t)-x_0(t)\big)^2 + b(b-a)\big(y(t)-x_0(t)\big)^2 + ab\big(x(t)-y(t)\big)^2 \Big) \nonumber
	\\
	& \le a(a-b) \Vert x-x_0 \Vert^2_{*, T} + b(b-a) \Vert y-x_0 \Vert^2_{*, T} + ab \Vert x-y \Vert^2_{*, T}.   \nonumber
\end{align}
Thus, applying \eqref{ineq: sup t} under an arbitrary coupling $\pi \in \Pi(\mu^\circ,\nu^\circ)$ and then minimizing over $\pi$, we obtain
\begin{equation}    \label{ineq: WOP}
	\Wop^2_2(\mu, \nu) \le (m_{\mu} - m_{\nu})^2 + (m_{\mu} - m_{\nu})\big(M_{x_0}(\mu) - M_{x_0}(\nu)\big) + m_{\mu}m_{\nu}\Wc_2^2(\mu^\circ,\nu^\circ).
\end{equation}
\begin{rem} \label{rem: equal mass WOP}
	If $m_{\mu}=m_{\nu}$, then \eqref{ineq: sup t} yields
	\begin{equation}    \label{eq: WOP}
		\Wop_2^2(\mu,\nu) = m_\mu^2 \Wc_2^2(\mu^\circ,\nu^\circ).
	\end{equation}
\end{rem}

\begin{thm} \label{thm:Polish}
	The space $(\Mc_{2,K}(\Cc^d), \Wop_2)$ is a Polish space.
\end{thm}

\begin{proof}
    We only sketch the proof, since it is a standard adaptation of the argument for	$\Mc_{2,K}(\R^d)$ in \cite{extendingWmetric}. Let $(\mu_n)_{n\in\N}$ be a Cauchy sequence in $(\Mc_{2,K}(\Cc^d),\Wop_2)$.	Then, $(m_{\mu_n})_{n\in\N}$ is Cauchy in $\R_+$ and $(T_{m_{\mu_n}}\#\mu_n^\circ)_{n\in\N}$ is Cauchy in $(\Pc_2(\Cc^d),\Wc_2)$. By completeness of $\Wc_2$, there exist $m\ge0$ and $\lambda\in\Pc_2(\Cc^d)$ such that
	\[
		m_{\mu_n}\to m, \qquad \Wc_2(T_{m_{\mu_n}}\#\mu_n^\circ,\lambda) \to 0.
	\]
	If $m>0$, define $\mu:=mT_m^{-1}\#\lambda$. Then $\mu_n\to\mu$ in $\Wop_2$ and $\mu\in\Mc_{2,K}(\Cc^d)$; the bound $M_{x_0}(\mu)\le Km$ follows from Lemma \ref{lem : villani} and the uniform bound $M_{x_0}(\mu_n)\le Km_{\mu_n}$.

	If $m=0$, then $\Wc_2^2(T_{m_{\mu_n}}\#\mu_n^\circ,\delta_{x_0}) \le K m_{\mu_n}^2 \to 0$ so necessarily $\lambda=\delta_{x_0}$, and hence $\mu_n\to 0$ in $\Wop_2$.

	Separability follows from separability of $(\Pc_2(\Cc^d),\Wc_2)$ together with the density of $\Q_+$ in $\R_+$, exactly as in the standard argument.
\end{proof}

\subsection{Fubini extension}

The graphon particle system \eqref{eq : graphon particle system with common noise} consists of the continuum of particles $\{X_u\}_{u \in I}$, as well as Brownian motions $\{B_u\}_{u \in I}$ with $I := [0, 1]$. Since the mapping $\Omega \times I \ni (\omega, u) \mapsto B_u(\omega)$ is not jointly measurable in the usual continuum product space with Lebesgue space $(I, \mathcal{B}_I, \lambda)$ on the index space $I$, an application of the rich Fubini extension, originally introduced in \cite{SUN200631}, is necessary. In this subsection, we shall only provide some essential parts of the Fubini extension for a formal definition of the graphon particle system \eqref{eq : graphon particle system with common noise}, as similar graphon particle settings were developed in \cite{Aurell2022, nonlinear:graphon}.

\begin{defn}
	Consider a Polish space $S$ and two probability spaces $(\Omega, \mathcal{F}, \mathbb{P})$ and $(I, \mathcal{I}, \lambda)$. A process $X : \Omega \times I \to S$ is said to be \textit{essentially pairwise independent (e.p.i.)}, if for $\lambda$-a.e. $u \in I$ and $\lambda$-a.e. $v \in I$, the random variables $X_u : \omega \mapsto X_u(\omega)$ and $X_v : \omega \mapsto X_v(\omega)$ are independent. A probability space $(\Omega \times I, \mathcal{F} \boxtimes \mathcal{I}, \mathbb{P} \boxtimes \lambda)$, extending the usual product space $(\Omega \times I, \mathcal{F} \otimes \mathcal{I}, \mathbb{P} \otimes \lambda)$ is said to be a \textit{Fubini extension}, if for any $(\mathbb{P} \boxtimes \lambda)$-integrable function $X$ on $(\Omega \times I, \mathcal{F} \boxtimes \mathcal{I})$
	\begin{enumerate}[label=(\roman*)]
		\item the two functions $X_u : \omega \mapsto X_u(\omega)$ and $X(\omega) : u \mapsto X_u(\omega)$ are integrable on $(\Omega, \mathcal{F}, \mathbb{P})$ for $\lambda$-a.e. $u \in I$ and on $(I, \mathcal{I}, \lambda)$ for $\mathbb{P}$-a.e. $\omega \in \Omega$, respectively;
		\item $\int_{\Omega} X_u(\omega) d\mathbb{P}$ and $\int_I X_u(\omega) d\lambda(u)$ are integrable on $(I, \mathcal{I}, \lambda)$ and $(\Omega, \mathcal{F}, \mathbb{P})$, respectively, with $\int_{\Omega \times I} X d(\mathbb{P} \boxtimes \lambda) = \int_I (\int_{\Omega} X_u(\omega) d\mathbb{P}) d\lambda(u) = \int_{\Omega} (\int_I X_u(\omega) d\lambda(u) ) d\mathbb{P}$.
	\end{enumerate}
\end{defn}

The following result shows the existence of the Fubini extension.

\begin{lem} [Theorem 1 of \cite{sun2009individual}]    \label{lem: Fubini extension}
	Let $S$ be a Polish space and $(\Omega, \mathcal{F}, \mathbb{P})$ be a probability space. There exist a probability space $(I, \mathcal{I}, \lambda)$ extending the Lebesgue space $(I, \mathcal{B}_I, \lambda)$ and a Fubini extension $(\Omega \times I, \mathcal{F} \boxtimes \mathcal{I}, \mathbb{P} \boxtimes \lambda)$ such that for any measurable mapping $\phi$ from $(I, \mathcal{I}, \lambda)$ to $\mathcal{P}(S)$, there is a $(\mathcal{F} \boxtimes \mathcal{I})$-measurable process $X: \Omega \times I \to S$ such that the random variables $X_u$ are e.p.i. and $\mathbb{P} \circ X_u^{-1} = \phi(u)$ for $u \in I$.
\end{lem}

In order to construct a space containing the continuum of (jointly measurable) Brownian motions $\{B_u\}_{u \in I}$ and initial distributions $\{X_u(0)\}_{u \in I}$ that are independent (in the sense of e.p.i.) in the graphon system, 
let $\phi : I \to \mathcal{P}(\Cc^d \times \R^d)$ be the mapping defined by $\phi(u) = w_u \otimes \rho_u(0)$ with the Wiener measure $w_u$ on $\Cc^d$ and a $\mathcal{B}_I$-measurable (thus $\mathcal{I}$-measurable) function $\rho_\cdot(0) : I \to \mathcal{P}(\R^d)$. 
Then, from Lemma \ref{lem: Fubini extension}, there exists a $(\mathcal{F} \boxtimes \mathcal{I})$-measurable process 
$Z : \Omega \times I \to \Cc^d \times \R^d$, 
defined by $Z_u(\omega) = (B_u(\omega), \zeta_u(\omega))$ 
for $u \in I$, such that $(B_u, \zeta_u)_{u \in I}$ are e.p.i. random variables 
and the law of $(B_u, \zeta_u)$ is equal to $w_u \otimes \rho_u(0)$ for every $u \in I$.

Following the notations of \cite{nonlinear:graphon}, we shall denote $L^2_{\boxtimes}(\Omega \times I, \Cc^d)$ the space of equivalent classes of $(\mathcal{F} \boxtimes \mathcal{I}, \mathcal{B}(\Cc^d))$-measurable functions that are $(\mathbb{P} \boxtimes \lambda)$-square integrable, i.e., $\phi \in L^2_{\boxtimes}(\Omega \times I, \Cc^d)$ if 
\begin{equation*}
	\E^{\boxtimes} \bigg[\sup_{0 \le t \le T} |\phi(t)|^2\bigg] := \int_{\Omega \times I} \sup_{0 \le t \le T} |\phi_u(t, \omega)|^2 \mathbb{P} \boxtimes \lambda (d\omega, du) < \infty,
\end{equation*}
and $L^2_{\lambda}(I, \R^d)$ for the Hilbert space of $\lambda$-a.e. equivalent classes of $\lambda$-measurable functions $\varphi: I \ni u \mapsto \varphi_u \in \R^d$ such that $\int_I |\varphi_u|^2 \lambda (du) < \infty$.

We now consider for $d, m \in \N$
\begin{align*}
	&\bm{\zeta} : \Omega \times I \ni (\omega, u) \mapsto \zeta_u(\omega) \in \R^d,
	\\
	&\bm{B} : \Omega \times I \ni (\omega, u) \mapsto \big(B_u(t, \omega)\big)_{t \in [0, T]} \in \Cc^d,
	\\
	&W : \Omega \ni \omega \mapsto  \big(W(t, \omega)\big)_{t \in [0, T]} \in \Cc^m,
\end{align*}
and the filtration $\mathbb{F}$ generated by $\bm{\zeta}$, $\bm{B}$, and $W$. For a candidate solution $\bm{X}$ and for each $u \in I$, $t \in [0,T]$, we write
\[
    \mu_u(t) := \mathcal{L}\big(X_u(t)\mid W\big) = \mathcal{L}\big(X_u(t)\mid \{W(s)\}_{0\le s\le t}\big)
\]
for the conditional law of $X_u(t)$ given the common noise. Then, the graphon particle system \eqref{eq : graphon particle system with common noise} can be formally written as a stochastic differential equation in $L^2_{\lambda}(I, \R^d)$
\begin{align}
	d\bm{X}(t) = \bm{\phi}_p \big( \bm{X}(t), (\mu_u(t))_{u \in I} \big) dt &+ \bm{\phi}_b \big( \bm{X}(t), (\mu_u(t))_{u \in I} \big) d\bm{B}(t)           \label{eq : formal graphon system}
	\\
	&+ \bm{\phi}_w \big( \bm{X}(t), (\mu_u(t))_{u \in I} \big) dW(t), \qquad t \in [0, T],      \nonumber
\end{align}
with the initial distribution $\bm{X}(0) = \bm{\zeta}$ and 
\begin{align*}
	\bm{\phi}_p &= (\phi_{p, u})_{u \in I} : L^2_{\lambda}(I, \R^d) \times \mathcal{P}_2(\R^d)^I \to L^2_{\lambda}(I, \R^d),
	\\
	\bm{\phi}_b &= (\phi_{b, u})_{u \in I} : L^2_{\lambda}(I, \R^d) \times \mathcal{P}_2(\R^d)^I \to L^2_{\lambda}(I, \R^{d \times d}),
	\\
	\bm{\phi}_w &= (\phi_{w, u})_{u \in I} : L^2_{\lambda}(I, \R^d) \times \mathcal{P}_2(\R^d)^I \to L^2_{\lambda}(I, \R^{d \times m}),
\end{align*}
where for $\bm{x} = (x_u)_{u \in I} \in L^2_{\lambda}(I, \R^d)$ and $\bm{\mu} \in \mathcal{P}_2(\R^d)^I$ the functions are defined by
\begin{align*}
	\phi_{\alpha, u} (\bm{x}, \bm{\mu}) &= \phi_{\alpha} \Big(x_u, \int_I G_{u, v} \mu_v dv \Big),
	\qquad \text{for} \quad \alpha = p, b, w, \quad \text{and} \quad \lambda-\text{a.e.} \quad u \in I.
\end{align*}

A solution to \eqref{eq : formal graphon system} is defined by an $\mathbb{F}$-progressively measurable process $\bm{X} \in L^2_{\boxtimes}(\Omega \times I, \Cc^d)$ satisfying for $u \in I$ and $(\mathbb{P} \boxtimes \lambda)$-a.e.
\begin{align*}
	X_u(t) = \zeta_u &+ \int_0^t \phi_{p, u} \big(\bm{X}(s), (\mu_u(s))_{u \in I} \big) ds + \int_0^t \phi_{b, u} \big(\bm{X}(s), (\mu_u(s))_{u \in I} \big) dB_u(s)
	\\
	&+ \int_0^t \phi_{w, u} \big(\bm{X}(s), (\mu_u(s))_{u \in I} \big) dW(s), \qquad \qquad t \in [0, T].
\end{align*}
We refer to Section 3 and Appendix A of \cite{nonlinear:graphon} for more details about the formal formulation of the graphon system.

\subsection{Notations and assumptions}
For any Polish space $(E,d)$, we denote by $C_b(E)$ the space of bounded continuous functions on $E$, and define $\langle f,\mu\rangle := \int_E f(x)\, \mu(dx)$ for any $f \in C_b(E)$, positive measure $\mu \in \Mc(E)$.

We first recall the graphon particle system \eqref{eq : graphon particle system with common noise} with its approximating $n$-particle system \eqref{eq : finite particle system with common noise}, and the following measures in $\mathcal{P}(\mathcal{C}^d)$:
\begin{equation}    \label{def: measures}
	\mu^n := \frac{1}{n} \sum_{i=1}^n \delta_{X^n_i}, \qquad
	\bar{\mu}^n := \frac{1}{n} \sum_{i=1}^n \delta_{X_{i/n}}, \qquad
	\mu_u := \mathcal{L}[X_u | W], \qquad
	\bar{\mu} := \int_0^1 \mu_u du,
\end{equation}
such that for each $t \in [0, T]$ we have the measures in $\mathcal{P}(\R^d)$
\begin{align*}
	\mu^n(t) := \frac{1}{n} \sum_{i=1}^n \delta_{X^n_i(t)}, \qquad 
	&\bar{\mu}^n(t) := \frac{1}{n} \sum_{i=1}^n \delta_{X_{i/n}(t)},
	\\
	\mu_u(t) := \mathcal{L}[X_u(t) | W], \qquad
	&\bar{\mu}(t) := \int_0^1 \mu_u(t) du.
\end{align*}
Note that $\mathcal{L}[X_u(t) | W] = \mathcal{L}[X_u(t) | \{W_s\}_{s \in [0,T]}] = \mathcal{L}[X_u(t) | \{W_s\}_{s \in [0,t]}]$.
Here and in what follows, the integral of measure should also be understood as a measure, as
\begin{align*}
	\bar{\mu}(A) &:= \int_0^1 \mu_u(A)du, \qquad \forall \, A \in \mathcal{B}(\Cc^d),
	\\
	\bar{\mu}(t)(A) &:= \int_0^1 \big(\mu_u(t)(A)\big) du, \qquad \forall \, A \in \mathcal{B}(\R^d).
\end{align*}

To simplify notations, we shall also write the graphon-weighted interaction measures
\begin{equation}    \label{def: simple notations}
	\mu_u^G := \int_0^1 G_{u,v}\mu_v\,dv, \qquad \mu_u^G(t) := \int_0^1 G_{u,v}\mu_v(t)\,dv,
	\qquad u \in [0,1], ~ t \in [0,T].
\end{equation}

We now list the following conditions that we impose on the two systems \eqref{eq : graphon particle system with common noise} and \eqref{eq : finite particle system with common noise}.

\begin{assumption}  \label{ass:Lipschitz}
	In the two systems \eqref{eq : graphon particle system with common noise} and \eqref{eq : finite particle system with common noise}, we assume that
	\begin{enumerate} [label=(\roman*)]
		\item the functions $\phi_{\alpha}$ for $\alpha = p, b, w$ are bounded and Lipschitz continuous, i.e., there exists some constant $L > 0$ such that for all $x_1, x_2 \in \R^d$ and $\mu_1, \mu_2 \in \mathcal{M}_{2,K}(\R^d)$,
		\begin{equation*}
			|\phi_\alpha(x_1, \mu_1) - \phi_\alpha(x_2, \mu_2)| \le L \Big(\max\{m_{\mu_1},m_{\mu_2},1\}|x_1 - x_2| + \Wop_2(\mu_1, \mu_2) \Big),
		\end{equation*}
		
		\item the mapping $[0, 1] \ni u \mapsto \mathcal{L}\big(X_u(0)\big) \in \mathcal{P}(\R^d)$ is measurable, and the initial distribution has uniformly bounded $(2+\eps)$ moment for some $\eps > 0$:
		\begin{equation*}
			\sup_{u \in [0, 1]} \E|X_u(0)|^{2 + \eps} < \infty.
		\end{equation*}
	\end{enumerate}
\end{assumption}

Lemma \ref{lem: unique pathwise solution} establishes well-posedness of the system \eqref{eq : graphon particle system with common noise}. It can be proven by a standard fixed-point (contraction) argument, so we omit the proof. Instead, we refer to the proof of Proposition 3.1 of \cite{nonlinear:graphon}, as adding a common noise does not change much of this proof.

\begin{lem} \label{lem: unique pathwise solution}
	Under Assumption \ref{ass:Lipschitz}, there exists a unique pathwise solution $\{X_u\}_{u\in[0, 1]}$ to the graphon particle system \eqref{eq : graphon particle system with common noise} with common noise. Furthermore, $\sup_{u \in [0, 1]} \E \Vert X_u \Vert^{2+\eps}_{*, T} < \infty$ holds, and the mapping $[0, 1] \ni u \mapsto \mathcal{L}(X_u) \in \mathcal{P}_2(\mathcal{C}^d)$ is measurable.
\end{lem}

The following additional assumptions are necessary when proving the $\Wc_{2}$-uniform continuity in $u$ of the conditional law $\mu_u$, defined in \eqref{def: measures} (see Lemma \ref{lem: continuity of average measure} below).

\begin{assumption}\label{ass: G Lipschitz}
	Regarding the two systems \eqref{eq : graphon particle system with common noise} and \eqref{eq : finite particle system with common noise}, we shall assume one of the following conditions on the graphon $G$. The first condition is stronger than the second one.
	\begin{enumerate}[label=(\roman*)]
		\item $G$ is Lipschitz continuous in the sense that for each $v \in [0,1]$ there exists some $L_v$ satisfying
		\begin{equation*}
			|G_{u_1,v} - G_{u_2,v}|
			~ \le ~
			L_v|u_1 - u_2|,
		\end{equation*}
		and $v \longmapsto L_v$ is $L^2$-integrable, i.e.,
		$\int_{0}^{1} |L_v|^2 dv < \infty$;
		\item $G$ is continuous on its domain $[0, 1] \times [0, 1]$.
	\end{enumerate}
\end{assumption}

\begin{assumption}  \label{ass: initial Wasserstein distance}
	In the two systems \eqref{eq : graphon particle system with common noise} and \eqref{eq : finite particle system with common noise}, we assume that the mapping $u \longmapsto \mu_u(0)$ satisfies the following continuity in $\Wc_{2+\eps}$ distance for some $\eps > 0$:
	\begin{equation}    \label{initial condition}
		\E\big[\Wc^{2+\eps}_{2+\eps} \big(\mu_{u_1}(0), \mu_{u_2}(0)\big)\big] \le \kappa |u_1-u_2|^{1 + \frac{\eps}{2}}, \qquad \forall \, u_1, u_2 \in [0, 1].
	\end{equation}
\end{assumption}

\begin{rem}
	The value of the constant $\epsilon$ in Assumptions \ref{ass:Lipschitz} and \ref{ass: initial Wasserstein distance} can be different, but we use the same symbol for simplicity. Moreover, the assumption \eqref{initial condition} can be relaxed to $\eps=0$ in the linear case (see Appendix \ref{sec: linear case}).
\end{rem}

\begin{assumption}\label{ass:G^n_and_G}
	For the graphons $G$ and $G^n$ in the two systems \eqref{eq : graphon particle system with common noise} and \eqref{eq : finite particle system with common noise}, we assume
	\begin{align*}
		\lim_{n \to \infty}\|G- G^n\|_\square
		~ = ~
		0,
	\end{align*}
	where the cut norm $\|\cdot\|_\square$ of a graphon is defined by 
	$$\|G\|_\square := \sup_{S,T \in \mathcal{B}(I)}\Bigg|\iint_{S \x T} G_{u,v}\,du\,dv\Bigg|.$$
	Here, $\mathcal{B}(I)$ is the set of all Borel measurable subsets of $[0, 1]$.
\end{assumption}

\bigskip

\section{Results}

This section states our main results regarding the two systems \eqref{eq : graphon particle system with common noise} and \eqref{eq : finite particle system with common noise}. Their proofs will be given in the next section.

\subsection{Preliminaries}

Before we state and prove our main theorems, we provide some preliminary results in this subsection. First, we show that the conditional law $\mu_u$ has a modification that is $\Wc_2$-uniformly continuous in $u$.

\begin{defn}    \label{Def: spatial modi}
	For the family of $\mathcal{P}(\mathcal{C}^d)$-valued random variables $\{\mu_u\}_{u \in [0, 1]}$ in \eqref{def: measures}, a family $\{\tilde{\mu}_u\}_{u \in [0, 1]}$ of $\mathcal{P}(\mathcal{C}^d)$-valued random variables satisfying
	\[
		\P\big(\mu_u = \tilde{\mu}_u\big) = 1, \qquad \forall\, u \in [0,1],
	\]
	is called a \textit{spatial modification} of $\{\mu_u\}_{u \in [0, 1]}$.
\end{defn}

\begin{lem} \label{lem: basic coefficient estimate}
	Let $q \in \{2,\,2+\eps\}$. Under Assumption~\ref{ass:Lipschitz}, for each $\alpha \in \{p,b,w\}$, all $x_1,x_2 \in \R^d$, and $\mu_1,\mu_2 \in \Mc_{2,K}(\R^d)$, there exists a constant $C_q > 0$ such that
	\begin{equation*}
		\big|\phi_\alpha(x_1,\mu_1)-\phi_\alpha(x_2,\mu_2)\big|^q
        \le C_q\Big( (1 \vee m_{\mu_1} \vee m_{\mu_2})^q |x_1-x_2|^q + \Wop_2^q(\mu_1,\mu_2) \Big).
	\end{equation*}
	In particular, if $m_{\mu_1},m_{\mu_2}\le 1$, then
	\begin{equation*}
		\big|\phi_\alpha(x_1,\mu_1)-\phi_\alpha(x_2,\mu_2)\big|^q
		\le C_q\Big( |x_1-x_2|^q+\Wop_2^q(\mu_1,\mu_2) \Big).
	\end{equation*}
\end{lem}

\begin{proof}
	The result follows directly from Assumption \ref{ass:Lipschitz} and $(a+b)^q \le 2^{q-1}(a^q+b^q)$.
\end{proof}

\begin{lem} \label{lem: continuity of average measure}
	If Assumptions \ref{ass:Lipschitz} and \ref{ass: G Lipschitz} (i) hold true, then the family $\{\mu_u\}_{u \in [0, 1]}$ of $\mathcal{P}_2(\mathcal{C}^d)$-valued random variables has a spatial modification $\{\tilde{\mu}_u\}_{u \in [0, 1]}$ and there exists an event $\Omega_0 \in \mathcal{F}$ with $\P(\Omega_0)=1$ such that, for every $\omega \in \Omega_0$, the mapping
	\[
		[0,1] \ni u \longmapsto \tilde{\mu}_u(\omega) \in \mathcal{P}_2(\mathcal{C}^d)
	\]
	is $\gamma$-H\"older continuous in $\Wc_2$ for $\gamma < \frac{\epsilon}{2(2+\epsilon)}$. In particular, for every $\omega \in \Omega_0$, the mapping $u \longmapsto \tilde{\mu}_u(\omega)$ is uniformly continuous in $\Wc_2$.
\end{lem}

The proof of Lemma \ref{lem: continuity of average measure} can be found in Appendix \ref{sec: proofs}.

\begin{rem}
	Throughout this paper, thanks to Lemma \ref{lem: continuity of average measure}, we shall abuse the notation by denoting the spatially modified version $\tilde{\mu}_u$ simply as $\mu_u$. Thus, almost surely, the mapping $u \longmapsto \mu_u(\omega)$ is $\mathcal{W}_2$-uniformly continuous in $u$.
\end{rem}

The following result provides several equivalent conditions for the $\mathcal{W}_2$-convergence of a sequence of probability measures. Since it will be used in the proofs of results in subsequent sections, we include it here for completeness.

\begin{lem}[Theorem 7.12 of \cite{villani}] \label{lem : villani}
	For a Polish space $X$, endowed with a distance $d$, let $(\nu_n)_{n \in \N}$ be a sequence of probability measures in $\mathcal{P}_2(X)$ and let $\nu \in \mathcal{P}(X)$. Then, the following statements are equivalent:
	\begin{enumerate}[label=(\roman*)]
		\item $\Wc_2(\nu_n, \nu) \xlongrightarrow{n \to \infty} 0$.
		\item $\nu_n \xlongrightarrow{n \to \infty} \nu$ in weak sense, and $(\nu_n)_{n \in \N}$ satisfies the tightness condition: for some $x_0 \in X$, 
		\[
		\lim_{a \to \infty} \limsup_{n \to \infty} \int_{d(x_0, x) \ge a} d(x_0, x)^2 d\nu_n(x) = 0.
		\]
		\item $\nu_n \xlongrightarrow{n \to \infty} \nu$ in a weak sense, and there is a convergence of the moment of order $2$: for some $x_0 \in X$,
		\[
		\int d(x_0, x)^2 d\nu_n(x) \xlongrightarrow{n \to \infty} \int d(x_0, x)^2 d\nu(x).
		\]
		\item For any continuous function $\varphi$ on $X$ satisfying the condition $|\varphi(x)| \le C[1+d(x_0, x)^2]$ for some $x_0 \in X$, 
		\[
		\int \varphi \, d\nu_n \xlongrightarrow{n \to \infty} \int \varphi \, d\nu.
		\]
	\end{enumerate}
\end{lem}

\subsection{Convergence within the graphon system}
This subsection is devoted to convergence results within the graphon mean-field system \eqref{eq : graphon particle system with common noise}, where dynamics in finite particle systems are not involved. In particular, the convergences of the empirical measures $\{\frac{1}{n} \sum_{i=1}^n \delta_{X_{i/n}}\}_{n \in \N}$ in $\Wc_2$, and the weighted empirical measures $\{\frac{1}{n} \sum_{i=1}^n G^n_{j,i}\delta_{X_{i/n}}\}_{n \in \N}$ in $\Wop_2$, respectively, are established. 

By recalling the measures in $\mathcal{P}(\mathcal{C}^d)$ of \eqref{def: measures}, we shall consider joint laws $\mathcal{L} (\bar{\mu}^n, \bar{\mu})$ and $\mathcal{L} (\bar{\mu}, \bar{\mu})$ on $\mathcal{P}(\mathcal{C}^d) \x \mathcal{P}(\mathcal{C}^d)$, instead of $\mathcal{L} (\bar{\mu}^n)$, $\mathcal{L} (\bar{\mu})$ on $\mathcal{P}(\mathcal{C}^d)$:
\begin{align}
	\mathcal{L} (\bar{\mu}^n, \bar{\mu}) (A) &:= \P [ (\bar{\mu}^n, \bar{\mu}) \in A], \qquad \forall \, A \in \mathcal{B}\big( \mathcal{P}(\mathcal{C}^d) \x \mathcal{P}(\mathcal{C}^d)\big),
	\\
	\mathcal{L} (\bar{\mu}, \bar{\mu}) (A) &:= \P [ (\bar{\mu}, \bar{\mu}) \in A], \qquad \quad \forall \, A \in \mathcal{B}\big( \mathcal{P}(\mathcal{C}^d) \x \mathcal{P}(\mathcal{C}^d)\big).
\end{align}

The following convergence \eqref{conv: W2 ppcd} will play a central role in proving subsequent results, and its proof is provided in Appendix \ref{sec: proofs}.

\begin{lem} \label{lem : W2 joint law conv}
	Under Assumptions \ref{ass:Lipschitz} and \ref{ass: G Lipschitz} (i), we have
	\begin{equation}    \label{conv: W2 ppcd}
		\lim_{n \to \infty} \W_2 \big( \Lc(\bar{\mu}^n,\bar{\mu}), \, \Lc(\bar{\mu},\bar{\mu}) \big) = 0.
	\end{equation}
\end{lem}

With the aid of Lemma \ref{lem : W2 joint law conv}, the following convergence of $\bar{\mu}^n$ to $\bar{\mu}$ can be established without difficulty. 

\begin{prop} \label{thm: bar mu_n to bar mu}
	Under Assumptions \ref{ass:Lipschitz} and \ref{ass: G Lipschitz} (i), we have
	\begin{equation}  \label{conv: mean W2 continuity}
		\E \Big[ \Wc^2_2( \bar{\mu}^n, \, \bar{\mu} ) \Big] \xlongrightarrow{n \to \infty} 0.
	\end{equation}
	Consequently, it follows that
	\begin{equation*}
		\Wc_2 \big( \Lc(\bar{\mu}^n), \, \Lc(\bar{\mu}) \big) \xlongrightarrow{n \to \infty} 0.
	\end{equation*}
\end{prop}

\begin{proof}
	In (i), (iv) of Lemma \ref{lem : villani}, we set
	\begin{equation*}
		X := \mathcal{P}(\mathcal{C}^d) \times \mathcal{P}(\mathcal{C}^d), \qquad \nu_n := \mathcal{L}(\bar{\mu}^n, \bar{\mu}), \qquad \nu := \mathcal{L}(\bar{\mu}, \bar{\mu}),
	\end{equation*}
	and define a continuous function $\varphi$ on $\mathcal{P}(\mathcal{C}^d) \times \mathcal{P}(\mathcal{C}^d)$
	\begin{align*}
		\varphi(\nu_1, \nu_2) := \Wc^2_2(\nu_1, \nu_2) \le 2 \Wc^2_2(\delta_0, \nu_1) +
		2 \Wc^2_2(\delta_0, \nu_2) = 2d^2 \big( (\delta_0,\delta_0), (\nu_1, \nu_2) \big)
	\end{align*}
	then, Lemma \ref{lem : W2 joint law conv} implies the result \eqref{conv: mean W2 continuity}:
	\begin{equation*}
		\E[\Wc^2_2(\bar{\mu}^n,\bar{\mu})] = \int \varphi \, d\nu_n \xlongrightarrow{n \to \infty} \int \varphi \, d\nu = 0.
	\end{equation*}
\end{proof}

By an approximation argument, the Lipschitz assumption on $G$ can be removed as the following.
\begin{thm} \label{thm: relaxed_bar mu_n to bar mu}
	Under Assumptions \ref{ass:Lipschitz} and \ref{ass: G Lipschitz} (ii), we have
	\begin{equation}
		\E \Big[ \Wc^2_2( \bar{\mu}^n, \, \bar{\mu} ) \Big] \xlongrightarrow{n \to \infty} 0.
	\end{equation}
\end{thm}

The proof approximating $G$ with a sequence of Lipschitz graphons is very similar to the argument in the proof of Theorem \ref{thm: relaxed_weighted mean W2 continuity} below, so we omit its proof.

In the following, we recall the notations \eqref{def: simple notations}. We can approximate the graphon-weighted average of the distributions of continuum of particles with the graphon-weighted average of the empirical measures of finite particles in the graphon mean field system \eqref{eq : graphon particle system with common noise} in terms of $\Wop_2$-metric. The proofs of these law of large number type results can be found in Appendix \ref{sec: proofs}.

\begin{prop} \label{thm: weighted mean W2 continuity}
	Under Assumptions \ref{ass:Lipschitz}, \ref{ass: G Lipschitz} (i), and \ref{ass:G^n_and_G},
	we have
	\begin{equation} \label{conv: weighted mean W2 continuity}
		\frac{1}{n}\sum_{i = 1}^n \E\Bigg[ \Wop^2_2\bigg(\frac{1}{n}\sum_{j = 1}^n G^n_{i,j} \delta_{X_{{j/n}}}, ~ \mu_{i/n}^G \bigg)\Bigg] \xlongrightarrow{n \to \infty} 0.
	\end{equation}
\end{prop}

Again by the approximation argument, the Lipschitz condition on the graphon can be removed.

\begin{thm}\label{thm: relaxed_weighted mean W2 continuity}
	Under Assumptions \ref{ass:Lipschitz}, \ref{ass: G Lipschitz} (ii), and \ref{ass:G^n_and_G}, we have
	\begin{equation}
		\frac{1}{n}\sum_{i = 1}^n \E\Bigg[ \Wop^2_2\bigg(\frac{1}{n}\sum_{j = 1}^n G^n_{i,j} \delta_{X_{{j/n}}}, ~ \mu_{i/n}^G \bigg)\Bigg] \xlongrightarrow{n \to \infty} 0.
	\end{equation}
\end{thm}

\begin{rem} \hspace{1pt}
	\begin{enumerate}[label=(\roman*)]
		\item Theorem \ref{thm: relaxed_bar mu_n to bar mu} can be seen as a special case of Theorem \ref{thm: relaxed_weighted mean W2 continuity}, if we simply let $G^n_{i,j} = 1$ for all $i,j \in [n]$, and $G_{u,v} = 1$ for all $u,v \in [0,1]$.
		\item To the best of our knowledge, this convergence in $\Wop_2$ is a new result, which reveals that the convergence of $\frac{1}{n}\sum_{j = 1}^n G^n_{i,j} \delta_{X_{{j/n}}}$ is actually not only in the topology of weak convergence but also in $\Wop_2$. The convergence in the topology of weak convergence can be found in Lemma 6.2 of \cite{BCW}.
	\end{enumerate}
\end{rem}

\subsection{Convergence of finite particle systems to the graphon system}
All the results in this subsection state the convergence of (weighted) empirical measures in finite particle systems to the graphon system \eqref{eq : graphon particle system with common noise}, as the number of particles $n$ tends to infinity, i.e., the convergence of $\{\frac{1}{n}\sum_{j = 1}^n G^n_{i,j}\delta_{X^n_j}\}_{n \in \N}$ in $\Wop_2$, and $\{\frac{1}{n}\sum_{j = 1}^n\delta_{X^n_j}\}_{n \in \N}$ in $\Wc_2$, respectively.

We first present a similar law of large number result between the graphon system \eqref{eq : graphon particle system with common noise} and the finite particle system \eqref{eq : finite particle system with common noise}. Note that this result is stronger than Theorem 4.1 of \cite{nonlinear:graphon}, in which they only consider the convergence of the integral of empirical marginal measures, i.e., $\frac{1}{n}\sum_{i = 1}^n \E\big[\int_0^T \Wop^2_2\big(\frac{1}{n}\sum_{j = 1}^n G^n_{i,j} \delta_{X^{{n}}_j(t)}, ~ \mu_{i/n}^G(t) \big)dt\big]$. We consider the convergence of the process law, which is stronger and does not require Assumption 3.1 (ii) in their paper. Their proofs are given in Appendix \ref{sec: proofs}.

\begin{thm} \label{thm: weighted mean W2 squared}
	Under Assumptions \ref{ass:Lipschitz}, \ref{ass: G Lipschitz} (ii), and \ref{ass:G^n_and_G}, we have
	\begin{equation}  \label{conv: weighted mean W2 squared}
		\frac{1}{n}\sum_{i = 1}^n \E\Bigg[ \Wop^2_2\bigg(\frac{1}{n}\sum_{j = 1}^n G^n_{i,j}\delta_{X^n_j}, \mu_{i/n}^G \bigg) \Bigg] \xlongrightarrow{n \to \infty} 0.
	\end{equation}
\end{thm}

We also provide the $\Wc_2$-convergence of the empirical measures of the $n$-particle system to those of the graphon system in the sense of process laws.

\begin{prop} \label{thm: W2 to zero}
	Under Assumptions \ref{ass:Lipschitz}, \ref{ass: G Lipschitz} (ii), and \ref{ass:G^n_and_G}, we have
	\begin{equation}    \label{conv: mean W2 squared}
		\E \Big[ \Wc^2_2( \mu^n, \bar{\mu} ) \Big] \xlongrightarrow{n \to \infty} 0,
	\end{equation}
	and
	\begin{equation}    \label{conv: to zero}
		\frac{1}{n} \sum_{i = 1}^n \E\bigg[\sup_{t \in [0,T]} \big|X^n_i(t) - X_{i/n}(t) \big|^2 \bigg]\xlongrightarrow{n \to \infty} 0.
	\end{equation}
\end{prop}
Note that the result \eqref{conv: to zero} recovers Theorem 4.1 of \cite{nonlinear:graphon} in the setting without common noise, but we used a different argument. Moreover, in the next theorem, we obtain a stronger result assuming $G$ is only measurable and showing convergence of process laws.

\begin{thm} \label{thm: relaxed_weighted mean W2 squared}
	Under Assumptions \ref{ass:Lipschitz} and \ref{ass:G^n_and_G}, and $G$ is only measurable instead of continuous, we have
	\begin{equation*}
		\E\Big[ \Wc^2_2\big(\mu^n, ~ \bar{\mu} \big) \Big] \xlongrightarrow{n \to \infty} 0.
	\end{equation*}
\end{thm}

\subsection{Convergence rates}
The following lemma will play a central role in proving Theorem \ref{thm: convergence rates} and its proof can be found in Appendix \ref{sec: proofs}.

\begin{lem} \label{lem:sampling_rate}
	Suppose Assumption \ref{ass:Lipschitz} holds true. Then, there exists some positive constant $C$ satisfying
	\begin{equation}\label{eq:rate_marginal}
		\sup_{t \in [0, T]} \sup_{i \in [n]} \E\Bigg[\Wop^2_2\bigg(\frac{1}{n}\sum_{j = 1}^n G^n_{i,j}\delta_{X_{{j/n}}(t)}, 
		\frac{1}{n}\sum_{j = 1}^n G^n_{i,j} \mu_{j/n}(t) \bigg)\Bigg]
		~ \le ~ 
		CM_n,
	\end{equation}
	where the rate of convergence $M_n$ is given by
	\begin{equation*}
		M_n = 
		\begin{cases}
			n^{-\frac{1}{2}} + n^{-\frac{\eps}{2+\eps}}, &~\mbox{if}~ d < 4, ~ \eps \neq 2, 
			\\
			n^{-\frac{1}{2}}\log(1 + n) + n^{-\frac{\eps}{2+\eps}}, &~\mbox{if}~ d = 4, ~ \eps \neq 2,
			\\
			n^{-\frac{2}{d}} + n^{-\frac{\eps}{2+\eps}}, &~\mbox{if}~ d > 4, ~ \eps \neq \frac{d-4}{d - 2}.
		\end{cases}
	\end{equation*}
	Here, $\eps$ is the constant in Assumption \ref{ass:Lipschitz} (ii).
\end{lem}

\begin{rem} \hspace{1pt}
	\begin{enumerate}[label=(\roman*)]
		\item The main idea and steps of the proof are almost the same as \cite{nonlinear:graphon}, which originates from Theorem 1 in \cite{FournierGuillin}, but our assumptions are weaker. Here, we provide the proof for completeness.
		
		\item It is natural to consider the convergence rate for the process law by replacing $X_{j/n}(t)$ and $\mu_{j/n}(t)$ by $X_{j/n}$ and $\mu_{j/n}$, respectively, in \eqref{eq:rate_marginal}. The main problem is that we don't know whether \eqref{eq:child_estimate} holds true for probability measures on $\Cc^d$, which is an infinite-dimensional space.
		In the proof of \eqref{eq:child_estimate}, one key feature of $\R^d$ is that we can divide $(-1,1]^d$ naturally into $2^{d}$ disjoint parts with half diameter, due to the total boundedness of $(-1,1]^d$. However, the unit ball in $\Cc^d$ doesn't have the total boundedness.
		We refer to Sections 2.1 and 2.2 of \cite{lei} for a detailed discussion.
	\end{enumerate}
\end{rem}

\begin{thm} \label{thm: convergence rates}
	Let Assumptions \ref{ass:Lipschitz} and \ref{ass: G Lipschitz} (i) hold true.
	Suppose that $G^n_{i,j} = G_{{i/n}, {j/n}}$ holds for $i, j \in [n], n \in \N$.
	Then, for some positive constant $K_{d,\eps} > 0$
	\begin{align*}
		\sup_{i \in [n]}\E\bigg[\sup_{t \in [0,T]}|X^n_i(t) - X_{{i/n}}(t)|^2\bigg]
		~ \le ~ 
		K_{d,\eps}M_n
	\end{align*}
	holds for $M_n$ given in Lemma \ref{lem:sampling_rate}.
\end{thm}

\begin{rem} \hspace{0.5pt}
	\begin{enumerate}[label=(\roman*)]
		\item Compared to the result \eqref{conv: to zero} in Proposition \ref{thm: W2 to zero}, it seems that the convergence in this theorem is stronger and at the same time with fewer assumptions. However, note that the convergence rate holds true only in the case of $G^n_{i,j} = G_{{i/n}, {j/n}}$, which is stronger than Assumption \ref{ass:G^n_and_G}.
		\item Compared to Theorem 4.2 in \cite{nonlinear:graphon}, we do not require Assumptions 3.1 and 3.2 in their paper.
	\end{enumerate}
\end{rem}

\bigskip

\begin{appendix}

\section{Proofs} \label{sec: proofs}

\begin{proof} [Proof of Lemma \ref{lem: continuity of average measure}]
	For fixed $u_1, u_2 \in [0, 1]$, consider auxiliary processes for $i = 1, 2$:
	\begin{align*}
		\tilde{X}_{u_i}(t) = \tilde{X}_{u_i}(0) + \int_0^t \phi_p \big(\tilde{X}_{u_i}(s), ~ \mu_{u_i}^G(s) \big) ds
		&+ \int_0^t \phi_b \big(\tilde{X}_{u_i}(s), ~ \mu_{u_i}^G(s) \big) dB(s)
		\\
        &+ \int_0^t \phi_w \big(\tilde{X}_{u_i}(s), ~ \mu_{u_i}^G(s) \big) dW(s),
	\end{align*}
	where $B$ is a $d$-dimensional Brownian motion independent of $W$ and $\{\tilde{X}_{u_i}(0)\}_{i=1, 2}$, and the conditional initial law is given as $\mathcal{L}(\tilde{X}_{u_i}(0)|W) = \mu_{u_i}(0)$ for $i=1, 2$. Then, from Lemma \ref{lem: unique pathwise solution}, we have $\mathcal{L}(\tilde{X}_{u_i}|W) = \mu_{u_i}$ for $i=1, 2$.
	
	Note that the positive measures $\mu^G_{u_i}(s)$ in $\mathcal{M}_2(\R^d)$ have their masses $m_{u_i} := \int_0^1 G_{u_i, v} dv \le 1$ for $i = 1, 2$. Using Jensen inequality with the convexity of $x \mapsto x^{2+\eps}$, H\"older inequality, and Burkholder-Davis-Gundy inequality, we obtain
	\begin{align}
		\E \left[ \Vert \tilde{X}_{u_1}-\tilde{X}_{u_2} \Vert^{2+\eps}_{*, t} \right]
		&\le C \E \left[ \big| \tilde{X}_{u_1}(0) - \tilde{X}_{u_2}(0) \big|^{2+\eps} \right]       \nonumber
		\\
		& \quad + C \E \left[ \left( \int_0^t \Big| \phi_{p} \big(\tilde{X}_{u_1}(s), ~ \mu_{u_1}^G(s) \big) - \phi_{p} \big(\tilde{X}_{u_2}(s), ~ \mu_{u_2}^G(s) \big) \Big| ds \right)^{2+\eps} \right]       \nonumber
		\\
		& \quad + C \sum_{\alpha = b, w} \E \left[ \left( \int_0^t \Big| \phi_{\alpha} \big(\tilde{X}_{u_1}(s), ~ \mu_{u_1}^G(s) \big) - \phi_{\alpha} \big(\tilde{X}_{u_2}(s), ~ \mu_{u_2}^G(s) \big) \Big|^2 ds \right)^{\frac{2+\eps}{2}} \right]        \nonumber
		\\
		&\le C \E \left[ \big| \tilde{X}_{u_1}(0) - \tilde{X}_{u_2}(0) \big|^{2+\eps} \right]       \nonumber
		\\
		& \quad + C t^{1+\eps} \E \left[ \int_0^t \Big| \phi_{p} \big(\tilde{X}_{u_1}(s), ~ \mu_{u_1}^G(s) \big) - \phi_{p} \big(\tilde{X}_{u_2}(s), ~ \mu_{u_2}^G(s) \big) \Big|^{2+\eps} ds \right]       \nonumber
		\\
		& \quad + C t^{\frac{\eps}{2}} \sum_{\alpha = b, w} \E \left[ \int_0^t \Big| \phi_{\alpha} \big(\tilde{X}_{u_1}(s), ~ \mu_{u_1}^G(s) \big) - \phi_{\alpha} \big(\tilde{X}_{u_2}(s), ~ \mu_{u_2}^G(s) \big) \Big|^{2+\eps} ds \right]        \nonumber
		\\
		&\le C \E \left[ \big| \tilde{X}_{u_1}(0) - \tilde{X}_{u_2}(0) \big|^{2+\eps} \right]   \label{ineq: u1 u2 bound}
		\\
		& \quad + C_t \sum_{\alpha = p, b, w} \E \left[ \int_0^t \Big| \phi_{\alpha} \big(X_{u_1}(s), ~ \mu_{u_1}^G(s) \big) - \phi_{\alpha} \big(X_{u_2}(s), ~ \mu_{u_2}^G(s) \big) \Big|^{2+\eps} ds \right].  \nonumber
	\end{align}
	Here and in what follows, $C$ is a positive constant depending only on $\eps$, and $C_t$ is a positive constant depending on $t$ and $\eps$, but their values may change from line to line. By Lemma \ref{lem: basic coefficient estimate}, we have
	\begin{align}
		\Big| \phi_{\alpha} \big(\tilde{X}_{u_1}(s), \mu_{u_1}^G(s) \big) & - \phi_{\alpha} \big(\tilde{X}_{u_2}(s), \mu_{u_2}^G(s) \big) \Big|^{2+\eps}  \label{ineq: 4th moment}
        \\
        & \qquad \le C \Big( \big|\tilde{X}_{u_1}(s)-\tilde{X}_{u_2}(s)\big|^{2+\eps} + \Wop_2^{2+\eps}\big(\mu_{u_1}^G(s), \mu_{u_2}^G(s)\big) \Big). \nonumber
	\end{align}
	
	To compute the $\Wop$ term, Lemma~\ref{Lem: WOP R^d} (ii) yields the inequality for any $s \in [0, T]$
	\begin{align*}
		\Wop_2^2\big(\mu_{u_1}^G(s), \mu_{u_2}^G(s)\big)
		\le (m_{u_1} - m_{u_2})^2 &+ |m_{u_1} - m_{u_2}| \cdot \big|M_0\big(\mu_{u_1}^G(s)\big) - M_0\big(\mu_{u_2}^G(s)\big)\big| \\
		&+ m_{u_1}m_{u_2}\Wc_2^2\big(\mu_{u_1}^{G\circ}(s), \mu_{u_2}^{G\circ}(s)\big). 
	\end{align*}
	For the first two terms on the right-hand side, Assumption \ref{ass: G Lipschitz} (i) implies
	\begin{equation*}
		(m_{u_1} - m_{u_2})^2 \le |u_1-u_2|^2 \int_0^1 L^2_v dv, \qquad |m_{u_1} - m_{u_2}| \le |u_1-u_2| \int_0^1 |L_v| dv,
	\end{equation*}
	and we have
	\begin{equation*}
		\E\big| M_{0}\big(\mu_{u_1}^G(s)\big) - M_{0}\big(\mu_{u_2}^G(s)\big) \big| \le 2 \sup_{u \in [0, 1]} \E \Vert X_u \Vert^2_{*, T} < \infty
	\end{equation*}
	from Lemma \ref{lem: unique pathwise solution}. For the last term, we discuss two cases.
	If $m_{{u_1}}= 0$ or $m_{u_2} = 0$, then
	\begin{align}
		\E[\Wop_2^2\big(\mu_{u_1}^G(s), \mu_{u_2}^G(s)\big)]
		&\le (m_{u_1} - m_{u_2})^2 + |m_{u_1} - m_{u_2}| \cdot \E\big|M_0\big(\mu_{u_1}^G(s)\big) - M_0\big(\mu_{u_2}^G(s)\big)\big|  \nonumber
		\\
		&\le C|u_1-u_2|.    \label{ineq: Wop u1u2 bound0}
	\end{align}
	Otherwise, Theorem 6.15 of \cite{villani2016optimal} yields
	\begin{equation}    \label{W bar bar bound}
		\Wc_2^2\big(\mu_{u_1}^{G\circ}(s), \mu_{u_2}^{G\circ}(s)\big) \le 2\int_{\R^d} |x|^2 \big|\mu_{u_1}^{G\circ}(s) - \mu_{u_2}^{G\circ}(s)\big|(dx)
		= 2 \int_0^1 \bigg| \frac{G_{u_1, v}}{m_{u_1}} - \frac{G_{u_2, v}}{m_{u_2}} \bigg| \int_{\R^d} |x|^2 \mu_v(s)(dx)dv.
	\end{equation}
	We now derive
	\begin{align*}
		m_{{u_1}}m_{{u_2}}\bigg| \frac{G_{u_1, v}}{m_{u_1}} - \frac{G_{u_2, v}}{m_{u_2}} \bigg|
		&= m_{{u_1}}m_{{u_2}}\bigg| \frac{G_{u_1, v}}{m_{u_1}} - \frac{G_{u_2, v}}{m_{u_1}} + \frac{G_{u_2, v}}{m_{u_1}} - \frac{G_{u_2, v}}{m_{u_2}} \bigg|
		\\ & \le m_{{u_2}}|G_{u_1, v}-G_{u_2, v}| + G_{u_2, v}|m_{u_2}-m_{u_1}|
		\\
		& \le |G_{u_1, v}-G_{u_2, v}| + |m_{u_2}-m_{u_1}|
		\le  |u_1-u_2| |L_v| + |u_1-u_2| \int_0^1 |L_v| dv,
	\end{align*}
	where the last inequality uses Assumption \ref{ass: G Lipschitz} (i). Thus, the inequality \eqref{W bar bar bound} becomes
	\begin{equation*}
		m_{{u_1}}m_{{u_2}}
		\E[\Wc_2^2\big(\mu_{u_1}^{G\circ}(s), \mu_{u_2}^{G\circ}(s)\big)]
		\le 2 |u_1-u_2| \Big(\sup_{u \in [0, 1]} \E \Vert X_u \Vert^2_{*, T} \Big) \int_0^1 |L_v|dv,
	\end{equation*}
	and then there exists $C>0$ that depends on the previous constants $K, \int_0^1|L_v|dv, \int_0^1|L_v|^2 dv$, and $\sup_{u \in [0, 1]} \E \Vert X_u \Vert^2_{*, T}$ satisfying
	\begin{equation}    \label{ineq: Wop u1u2 bound}
		\E[\Wop_2^2\big(\mu_{u_1}^{G}(s), \mu_{u_2}^{G}(s)\big)]
		\le C|u_1-u_2|, \qquad \forall \, s \in [0, T].
	\end{equation}
	Therefore, by combining \eqref{ineq: Wop u1u2 bound} (or \eqref{ineq: Wop u1u2 bound0}) with \eqref{ineq: 4th moment}, the inequality \eqref{ineq: u1 u2 bound} becomes for any $t \in [0, T]$
	\begin{align*}
		\E \left[ \Vert \tilde{X}_{u_1}-\tilde{X}_{u_2} \Vert^{2+\eps}_{*, t} \right]
		\le & C \E \left[ \big| \tilde{X}_{u_1}(0) - \tilde{X}_{u_2}(0) \big|^{2+\eps} \right]
		\\
		& \qquad + C_t \int_0^t \E \left[ \big|\tilde{X}_{u_1}(s) - \tilde{X}_{u_2}(s)\big|^{2+\eps} \right] ds + C_t |u_1-u_2|^{1 + \frac{\eps}{2}},
	\end{align*}
	and Grönwall's inequality yields
	\begin{align*}
		\E \left[ \Wc^{2+\eps}_2(\mu_{u_1},\mu_{u_2}) \right]
		&\le
		\E \left[ \Vert \tilde{X}_{u_1}-\tilde{X}_{u_2} \Vert^{2+\eps}_{*, T} \right]
		\le
		C \E \left[ \big| \tilde{X}_{u_1}(0) - \tilde{X}_{u_2}(0) \big|^{2+\eps} \right] + C|u_1-u_2|^{1 + \frac{\eps}{2}}
		\\
		&= C \E \left[ \E \big| \tilde{X}_{u_1}(0) - \tilde{X}_{u_2}(0) \big|^{2+\eps} \bigg| W\right]  + C|u_1-u_2|^{1 + \frac{\eps}{2}}.
	\end{align*}
	Here, the constant $C$ now depends on $T$. Taking the infimum over the conditional random variables satisfying $\mathcal{L}(X_{u_i}(0)|W) = \mu_{u_i}(0)$ for $i=1, 2$, we obtain
	\begin{equation}   \label{ineq: W^2_2}
		\E \left[ \Wc^{2+\eps}_2(\mu_{u_1},\mu_{u_2}) \right]
		\le C \E \left[ \Wc^{2+\eps}_{2+\eps} \big(\mu_{u_1}(0),\mu_{u_2}(0) \big) \right] + C|u_1-u_2|^{1 + \frac{\eps}{2}} \le C|u_1-u_2|^{1 + \frac{\eps}{2}}.
	\end{equation}
	Here, the last inequality uses Assumption \ref{ass: initial Wasserstein distance}, and note that $\mathcal{P}_2(\mathcal{C}^d)$ is complete and separable.

	Since $(\Pc_2(\Cc^d),\Wc_2)$ is complete and separable, the standard dyadic Kolmogorov-Chentsov/Borel-Cantelli argument applied to \eqref{ineq: W^2_2} yields a spatial modification $\{\tilde{\mu}_u\}_{u\in[0,1]}$ such that, for every $\gamma < \frac{\eps}{2(2+\eps)}$, there exists an event $\Omega_0 \in \mathcal{F}$ with $\P(\Omega_0)=1$ and
	\[
		\Wc_2\big(\tilde{\mu}_{u_1}(\omega),\tilde{\mu}_{u_2}(\omega)\big)
		\le C(\omega)|u_1-u_2|^\gamma, \qquad \forall\,u_1,u_2\in[0,1],\ \omega\in\Omega_0.
	\]
	For completeness, we briefly indicate the argument: one first obtains the corresponding estimate on the dyadic rationals by Markov's inequality and the Borel-Cantelli lemma, and then extends it to all of $[0,1]$ by completeness of $(\Pc_2(\Cc^d),\Wc_2)$.
	Finally, for each fixed $u \in [0, 1]$ and any dyadic sequence $u_n \to u$, Fatou's lemma and \eqref{ineq: W^2_2} imply
	\[
		\E \big[\Wc_2^{2+\eps}(\mu_u,\tilde{\mu}_u)\big]
		\le \liminf_{n\to\infty} \E \big[\Wc_2^{2+\eps}(\mu_u,\mu_{u_n})\big] = 0,
	\]
	so $\mu_u=\tilde{\mu}_u$ almost surely. This completes the proof.
\end{proof}

\medskip

\begin{proof} [Proof of Lemma \ref{lem : W2 joint law conv}]
	The proof consists of 3 parts.
	
	\textbf{Part 1:} $\{\Lc(\bar{\mu}^n,\bar{\mu})\}_{n \in \N}$ is tight in $\mathcal{P}(\mathcal{P}(\mathcal{C}^d) \x \mathcal{P}(\mathcal{C}^d))$. \\
	The tightness of $\{\Lc(\bar{\mu}^n,\bar{\mu})\}_{n \in \N}$ in $\mathcal{P}(\mathcal{P}(\mathcal{C}^d) \x \mathcal{P}(\mathcal{C}^d))$ is equivalent to the tightness of  $\{\Lc(\bar{\mu}^n)\}_{n \in \N}$ in $\mathcal{P}(\mathcal{P}(\mathcal{C}^d))$, which will hold true if and only if
	$\frac{1}{n}\sum_{i=1}^n \mathcal{L}(X_{i/n})$ is tight in $\mathcal{P}(\mathcal{C}^d)$ from Chapter I (2.5) in \cite{sznitman}.
	
	To show the latter condition, we first denote $\mathcal{T}$ the collection of all $\mathbb{F} = (\mathcal{F}_t)_{t \in [0, T]}$-stopping times. We obtain that there exists a constant $C>0$, independent of $n$, satisfying for all $i \in [n]$ and $n \in \N$
	\begin{align}
		& ~
		\sup_{\tau \in \mathcal{T}} \E \Big|X_{i/n}\big((\tau + \epsilon) \wedge T \big) - X_{i/n}(\tau) \Big|^2 \nonumber
		\\
		~ = & ~
		\sup_{\tau \in \mathcal{T}} \E \Big|
		\int_{\tau}^{(\tau + \varepsilon)\wedge T}
		\widetilde{\phi_p}(i/n) dt
		~ + ~
		\int_{\tau}^{(\tau + \varepsilon)\wedge T} \widetilde{\phi_b}(i/n) dB_{i/n}(t)
		~ + ~
		\int_{\tau}^{(\tau + \varepsilon)\wedge T} \widetilde{\phi_w}(i/n) dW(t) \Big|^2 \nonumber
		\\
		~ \le & ~
		C\sup_{\tau \in \mathcal{T}} \E \bigg[ 
		\int_{\tau}^{(\tau + \varepsilon)\wedge T}
		|\widetilde{\phi_p}(i/n)|^2+
		|\widetilde{\phi_b}(i/n)|^2+
		|\widetilde{\phi_w}(i/n)|^2 dt
		\bigg] 
		\le C \epsilon, \label{ineq: bounded sup stopping times}
	\end{align}
	where for each $u \in [0,1]$ and $\alpha = p, b, w$,
	\begin{align*}
		\widetilde{\phi_{\alpha}}(u)
		~ := ~
		\phi_{\alpha} \big(X_u(t), ~ \mu_{u}^G(t) \big).
	\end{align*}
    Moreover, Lemma \ref{lem: unique pathwise solution} and Markov's inequality yield
	\begin{align*}
		\sup_{n \in \N}\frac{1}{n}\sum_{i=1}^n \P\bigg[\sup_{t \in [0,T]} |X_{i/n}(t)| > a\bigg] &\le \frac{1}{a^{2+\eps}} \sup_{u \in [0,1]} E\big[\|X_u\|_{*,T}^{2+\eps}\big] \xlongrightarrow{a \to \infty} 0.
	\end{align*}
	Hence, the compact containment condition in Aldous's criterion (see, e.g., (16.22) of \cite{billingsley}) is satisfied. Together with \eqref{ineq: bounded sup stopping times}, this provides the tightness of $\frac{1}{n}\sum_{i=1}^n \mathcal{L}(X_{i/n})$ in $\mathcal{P}(\mathcal{C}^d)$ from Aldous's criterion (see, e.g., (16.22)-(16.23) of \cite{billingsley}).
	
	\smallskip
	
	\textbf{Part 2:} Weak convergence of 
	$ \Lc(\bar{\mu}^n,\bar{\mu})  \xlongrightarrow{n \to \infty} \Lc(\bar{\mu},\bar{\mu}) $. \\
	For each $m \in \N$, consider $(m + 1)$ many of arbitrary test functions $\{h_i\}_{i \in [m]} \subset C_b(\mathcal{C}^d)$, and $f \in C_b(\Cc^d)$. We shall show the convergence
	\begin{align}    
		\int_{\mathcal{P}(\mathcal{C}^d) \times \mathcal{P}(\mathcal{C}^d)}
		\prod_{i=1}^m \langle h_i, \nu \rangle \langle f,\nu' \rangle
		\Lc(\bar{\mu}^n,\bar{\mu}) &(d\nu, d\nu') 
		~~~\xlongrightarrow{n \to \infty}	\label{conv: test functions}
		\\ 
		&\int_{\mathcal{P}(\mathcal{C}^d) \times \mathcal{P}(\mathcal{C}^d)}
		\prod_{i=1}^m \langle h_i, \nu\rangle  \langle f,\nu' \rangle
		\Lc(\bar{\mu},\bar{\mu}) (d\nu, d\nu').		\nonumber
	\end{align}
	For simplicity, we consider the case $m = 2$, then the left-hand side of \eqref{conv: test functions} can be expressed as
	\begin{align*}
		& \quad \ \frac{1}{n^2}\sum_{i=1}^n \sum_{j=1}^n \E \Big[h_{1}(X_{i/n}) h_{2}(X_{j/n}) \langle f,\bar{\mu} \rangle\Big] 
		= \frac{1}{n^2} \sum_{i=1}^n \sum_{j=1}^n \E \Bigg[ \E \Big[ h_{1}(X_{i/n}) h_{2}(X_{j/n}) \big| W \Big] 
		\langle f,\bar{\mu} \rangle\Bigg]
		\\
		& = \frac{1}{n^2} \sum_{i=1}^n \E \Bigg[ \E \Big[ h_{1}(X_{i/n}) h_{2}(X_{i/n}) \big| W \Big] \langle f,\bar{\mu} \rangle \Bigg] 
		+ 
		\frac{2}{n^2} \sum_{i<j} \E \Bigg[ \E \Big[ h_{1}(X_{i/n}) h_{2}(X_{j/n}) \big| W \Big] \langle f,\bar{\mu} \rangle \Bigg]
		\\
		& = \frac{1}{n^2} \sum_{i=1}^n \E \Bigg[ \E \Big[ h_{1}(X_{i/n}) h_{2}(X_{i/n}) \big| W \Big] \langle f,\bar{\mu} \rangle\Bigg] + \frac{2}{n^2} \sum_{i<j} \E \Bigg[ \E \Big[ h_{1}(X_{i/n}) \big| W \Big] \E \Big[ h_{2}(X_{j/n}) \big| W \Big] \langle f,\bar{\mu} \rangle\Bigg]
		\\
		& = \frac{1}{n^2} \sum_{i=1}^n \E \Bigg[ \E \Big[ h_{1}(X_{i/n}) \big| W \Big] \E \Big[ h_{2}(X_{i/n}) \big| W \Big] \langle f,\bar{\mu} \rangle\Bigg] 
		+ \frac{2}{n^2} \sum_{i<j} \E \Bigg[ \E \Big[ h_{1}(X_{i/n}) \big| W \Big] \E \Big[ h_{2}(X_{j/n}) \big| W \Big] \langle f,\bar{\mu} \rangle\Bigg]
		\\
		& \qquad \qquad + \frac{1}{n^2} \sum_{i=1}^n \E \Bigg[ \bigg(\E \Big[ h_{1}(X_{i/n}) h_{2}(X_{i/n}) \big| W \Big] - \E \Big[ h_{1}(X_{i/n}) \big| W \Big] \E \Big[ h_{2}(X_{i/n}) \big| W \Big]\bigg)\langle f,\bar{\mu} \rangle \Bigg]
		\\
		&= \E \Bigg[ \bigg( \frac{1}{n} \sum_{i=1}^n \E \Big[ h_{1}(X_{i/n}) \big| W \Big] \bigg) \bigg( \frac{1}{n} \sum_{j=1}^n \E \Big[ h_{2}(X_{j/n}) \big| W \Big] \bigg) \langle f,\bar{\mu} \rangle \Bigg]
		\\
		& \qquad \qquad + \frac{1}{n^2} \sum_{i=1}^n \E \Bigg[ \bigg(\E \Big[ h_{1}(X_{i/n}) h_{2}(X_{i/n}) \big| W \Big] - \E \Big[ h_{1}(X_{i/n}) \big| W \Big] \E \Big[ h_{2}(X_{i/n}) \big| W \Big]\bigg)\langle f,\bar{\mu} \rangle \Bigg].
	\end{align*}
	Here, the third equality uses the conditional independence of $X^n_i$ and $X^n_j$ given $W$, whenever $i \neq j$. Since the functions $f, h_i$'s are bounded, say, by a constant $C>0$, the last term is bounded by
	\begin{equation*}
		\frac{1}{n^2} \sum_{i=1}^n \E \Bigg[ \bigg(\E \Big[ h_{1}(X_{i/n}) h_{2}(X_{i/n}) \big| W \Big] - \E \Big[ h_{1}(X_{i/n}) \big| W \Big] \E \Big[ h_{2}(X_{i/n}) \big| W \Big]\bigg)\langle f,\bar{\mu} \rangle \Bigg]
		\le 
		\frac{2C^3n}{n^2} = \frac{2C^3}{n}.
	\end{equation*}
	Taking limit $n \to \infty$, we can obtain the following convergence from Lemma \ref{lem: continuity of average measure}
	\begin{equation}    \label{conv into integral}
		\frac{1}{n} \sum_{i=1}^n \E \Big[ h_{1}(X_{i/n}) \big| W \Big] \xlongrightarrow{n \to \infty}  \int_0^1 \E \big[ h_1(X_u) | W \big] du,
		~\P~\mbox{-a.s.}
	\end{equation}
	
	Therefore, from the Dominated Convergence Theorem, we obtain the convergence
	\begin{align*}
		\frac{1}{n^2} &\sum_{i=1}^n \sum_{j=1}^n \E \Big[h_{1}(X_{i/n}) h_{2}(X_{j/n})\langle f,\bar{\mu} \rangle\Big] \xlongrightarrow{n \to \infty} \E \bigg[ \langle h_1,\bar{\mu} \rangle\langle h_2,\bar{\mu} \rangle \langle f,\bar{\mu} \rangle \bigg] 
	\end{align*}
	and the last limit is equal to the right-hand side of \eqref{conv: test functions}. The general case $m > 2$ can be proven in a similar manner.
	Then, by Theorem \ref{thm: characterization2}, we can conclude the proof of part 2.
	
	\smallskip
	
	\textbf{Part 3:} Uniform integrability of $\Lc(\bar{\mu}^n,\mub)$. \\
	We shall show that
	\begin{align*}
		\lim_{a \to \infty} \sup_{n \in \N} \, 
		\E \bigg[ 
		\Big(\Wc^2_2(\bar{\mu}^n, \delta_0) + \Wc^2_2(\bar{\mu}, \delta_0) \Big)
		\mathbbm{1}_{\big\{\sqrt{\Wc^2_2(\bar{\mu}^n, \delta_0) + \Wc^2_2(\bar{\mu}, \delta_0) } \ge a \big\}} \bigg] = 0.
	\end{align*}
	From Markov's inequality, we derive
	\begin{align*}
		&\E \bigg[ 
		\Big(\Wc^2_2(\bar{\mu}^n, \delta_0) + \Wc^2_2(\bar{\mu}, \delta_0)\Big)
		\mathbbm{1}_{\big\{\sqrt{\Wc^2_2(\bar{\mu}^n, \delta_0) + \Wc^2_2(\bar{\mu}, \delta_0) } \ge a \big\}} \bigg]
		\\
		&\qquad \qquad \qquad \qquad \qquad \qquad \le \frac{1}{a^{\eps}} \sup_{n \in \N} \, \E \Big[ \big(\Wc^2_2(\bar{\mu}^n, \delta_0) + \Wc^2_2(\bar{\mu}, \delta_0) \big)^{1 + \eps/2}\Big]
		\\
		& \qquad \qquad \qquad \qquad \qquad \qquad \le \frac{C}{a^{\eps}} \sup_{n \in \N} \, \bigg( \frac{1}{n} \sum_{i=1}^n \E \Big[ \sup_{t \in [0, T]} |X_{i/n}(t)|^{2 + \eps} \Big] \bigg)
		+ \frac{C}{a^{\eps}} \E \big[\Wc^{2 + \eps}_2(\bar{\mu}, \delta_0) \big].
	\end{align*}
	Similar to \eqref{ineq: bounded sup stopping times}, thanks to the boundedness of $\phi$'s, there exists a constant $C>0$ such that the first term on the right-hand side is bounded by $CT/a$. Sending $a \to \infty$ yields the desired uniform integrability.
	
	Lemma \ref{lem : villani} (ii) with the results in Parts 2 and 3 proves the result \eqref{conv: W2 ppcd}.
\end{proof}

\medskip

\begin{proof} [Proof of Proposition \ref{thm: weighted mean W2 continuity}]
	We first denote, for any $i \in [n]$ and $n \in \N$,
	\[
		\bar{\mu}^n_i := \frac{1}{n}\sum_{j = 1}^n G^n_{i,j}\delta_{X_{j/n}}, \qquad \nu_n := \frac{1}{n}\sum_{i = 1}^n \Lc(\bar{\mu}^n_i, \mu^G_{i/n}), \qquad \nu := \int_0^1 \Lc(\mu_u^G, \mu_u^G)\,du.
	\]
	As in the proof of Proposition \ref{thm: bar mu_n to bar mu}, it is sufficient to prove that
	\[
		\Wc_2(\nu_n,\nu) \xlongrightarrow{n \to \infty} 0.
	\]

	The proof follows the same three-step argument as in Lemma \ref{lem : W2 joint law conv}; we only indicate the modifications.

	\smallskip

	\noindent\textbf{Part 1: Tightness.}
	Since $0 \le G \le 1$, the masses of $\bar{\mu}^n_i$ and $\bar{\mu}_{i/n}$ are bounded by $1$, and their second moments are uniformly controlled by Lemma \ref{lem: unique pathwise solution}. Hence the families
	\[
		\bigg\{\frac{1}{n}\sum_{j=1}^n G^n_{i,j}\Lc(X_{j/n})\bigg\}_{n \in \N,\, i \in [n]} \quad\text{and}\quad \bigg\{\int_0^1 G_{i/n,v}\Lc(X_v)\,dv\bigg\}_{n \in \N,\, i \in [n]}
	\]
	are tight in $\Mc_{2,K}(\Cc^d)$ for some $K>0$, by the same Aldous-type argument as in Lemma~\ref{lem : W2 joint law conv}.

	\smallskip

	\noindent\textbf{Part 2: Weak convergence.}
	Testing against functions of the form
	\[
		(\nu,\nu') \longmapsto \prod_{j=1}^m \langle h_j,\nu\rangle \prod_{k=1}^\ell \langle f_k,\nu'\rangle, \qquad h_j,f_k \in C_b(\Cc^d),
	\]
	one obtains the same expansion as in Lemma \ref{lem : W2 joint law conv}, except that each unweighted average
	\[
		\frac{1}{n}\sum_{i=1}^n \E[h(X_{i/n})\mid W]
	\]
	is replaced by the weighted average
	\[
		\frac{1}{n}\sum_{i=1}^n G^n_{k,i}\E[h(X_{i/n})\mid W].
	\]
	The key additional convergence result we need is
	\[
		\frac{1}{n}\sum_{i=1}^n G^n_{k,i}\E[h(X_{i/n})\mid W] \xlongrightarrow{n \to \infty,\, k/n \to u} \int_0^1 G_{u,v}\E[h(X_v)\mid W]\,dv,
	\]
	which follows from Lemma \ref{lem: continuity of average measure}, Riemann-sum convergence, and Lemma 8.11 of \cite{lovasz}. 
    The auxiliary term for this convergence is
    \begin{align*}
        \frac{1}{n} \sum_{i=1}^n G_{k/n,i/n}\E \Big[ h_{1}(X_{{i/n}}) \big| W \Big].
    \end{align*}
    Therefore, $\nu_n \Rightarrow \nu$.

	\smallskip

	\noindent\textbf{Part 3: Uniform integrability.}
	The same argument as in Part 3 of Lemma \ref{lem : W2 joint law conv}, with $\Wc_2$ replaced by $\Wop_2$, yields the uniform integrability of $\{\nu_n\}_{n\in\N}$. Since $\Mc_{2,K}(\Cc^d)$ is Polish by Theorem~\ref{thm:Polish}, Lemma~\ref{lem : villani} implies that $\Wc_2(\nu_n,\nu)\to 0$. This proves the claim.
\end{proof}

\medskip

\begin{proof}  [Proof of Theorem \ref{thm: relaxed_weighted mean W2 continuity}]
	This proof consists of two parts.
	
	\smallskip
	
	{\noindent \textbf{Part 1:} Approximation of the graphon $G$.}
    
    Fix $\eta > 0$. Since $G$ is uniformly continuous on $[0,1]^2$, we may choose a Lipschitz graphon $\tilde G : [0,1]^2 \to [0,1]$ such that
    \[
    	\|G-\tilde G\|_{L^\infty} \le \eta.
    \]
    For each $n \in \N$, let $\mathcal P_n=\{I_i\times I_j\}_{i,j=1}^n$, where
    \[
    	I_i:=\Big(\frac{i-1}{n},\frac{i}{n}\Big], \qquad i \in [n],
    \]
    and let $P_n:L^1([0,1]^2)\to \mathcal S_n$ denote the conditional expectation onto the step functions that are constant on each atom of $\mathcal P_n$.
    Define $\tilde G^n := P_n\tilde G + G^n - P_nG$. Then, since $P_n$ is an $L^1$-contraction,
    \[
    	\sup_{n \in \N}\|G^n-\tilde G^n\|_{L^1} = \sup_{n \in \N}\|P_n(\tilde G-G)\|_{L^1} \le \|\tilde G-G\|_{L^1} \le \eta.
    \]
    Moreover,
    \[
    	\|\tilde G^n-\tilde G\|_\square \le \|P_n\tilde G-\tilde G\|_{L^1} + \|G^n-G\|_\square + \|P_nG-G\|_{L^1} \xlongrightarrow{n\to\infty} 0,
    \]
    because $P_n\tilde G \to \tilde G$ and $P_nG \to G$ in $L^1$, while $\|G^n-G\|_\square \to 0$ by Assumption \ref{ass:G^n_and_G}.
    Thus, the required approximation exists.
	
	\smallskip
	
	{\noindent \textbf{Part 2:} Estimation.}
	
	By Proposition \ref{thm: weighted mean W2 continuity}, it follows that
	\begin{equation*}
		\frac{1}{n}\sum_{i = 1}^n \E\Bigg[ \Wop^2_2\bigg(\frac{1}{n}\sum_{j = 1}^n \Gt^n_{i,j} \delta_{\Xt_{{j/n}}}, ~ \mut^{\Gt}_{i/n}\bigg)\Bigg] \xlongrightarrow{n \to \infty} 0,
	\end{equation*}
	where $\mut_v := \Lc(\Xt_v|W)$ and $\{\Xt_v\}_{v\in [0,1]}$ is the solution to the graphon system \eqref{eq : graphon particle system with common noise} with $\Gt$.
	By triangle inequality, it remains to estimate the terms
	\begin{align}\label{eq: distance between n particle}
		\frac{1}{n}\sum_{i = 1}^n \E\Bigg[ \Wop^2_2\bigg(\frac{1}{n}\sum_{j = 1}^n \Gt^n_{i,j} \delta_{\Xt_{{j/n}}}, ~ \frac{1}{n}\sum_{j = 1}^n G^n_{i,j} \delta_{X_{{j/n}}}\bigg)\Bigg],
		\\ \label{eq: distance between integral}
		\frac{1}{n}\sum_{i = 1}^n \E\bigg[ \Wop^2_2\Big(\mu^G_{i/n}, ~ \mut^{\Gt}_{i/n}\Big)\bigg].
	\end{align}
	Again, by triangle inequality, the term \eqref{eq: distance between n particle} is bounded by
	\begin{align*}
		&
		\frac{1}{n}\sum_{i = 1}^n \E\Bigg[ \Wop^2_2\bigg(\frac{1}{n}\sum_{j = 1}^n \Gt^n_{i,j} \delta_{\Xt_{{j/n}}}, ~ \frac{1}{n}\sum_{j = 1}^n G^n_{i,j} \delta_{X_{{j/n}}}\bigg)\Bigg]
		\\
		\le ~&
		\frac{2}{n}\sum_{i = 1}^n \E\Bigg[ \Wop^2_2\bigg(\frac{1}{n}\sum_{j = 1}^n \Gt^n_{i,j} \delta_{\Xt_{{j/n}}}, ~ \frac{1}{n}\sum_{j = 1}^n \Gt^n_{i,j} \delta_{X_{{j/n}}}\bigg)\Bigg]
		\\& ~~~~~~~+
		\frac{2}{n}\sum_{i = 1}^n \E\Bigg[ \Wop^2_2\bigg(\frac{1}{n}\sum_{j = 1}^n \Gt^n_{i,j} \delta_{X_{{j/n}}}, ~ \frac{1}{n}\sum_{j = 1}^n G^n_{i,j} \delta_{X_{{j/n}}}\bigg)\Bigg]
		\\
		\le ~ &
		\frac{2}{n^2}\sum_{i,j = 1}^n \Gt^n_{i,j}\E\bigg[\sup_{t \in [0,T]}|X_{j/n}(t) - \Xt_{j/n}(t)|^2\bigg] 
		+ \frac{2K}{n^2}\sum_{i,j = 1}^n|\Gt^n_{i,j} - G^n_{i,j}|
		\\
		\le ~ &
		2\sup_{v \in [0,1]}\E\bigg[\sup_{t \in [0,T]}|X_{v}(t) - \Xt_{v}(t)|^2\bigg] 
		~+~
		2K\|G^n - \Gt^n\|_{L^1},
	\end{align*}
	where the second inequality holds by estimate for Wasserstein distance between empirical measures and similar argument of $\Gc^n_i$ in Theorem \ref{thm: convergence rates}, for some constant $K$, which may vary from line to line.
	
	Similarly, the term \eqref{eq: distance between integral} is bounded by
	\begin{align*}
		&
		\frac{1}{n}\sum_{i = 1}^n \E\bigg[ \Wop^2_2\Big(\mu^G_{i/n}, ~ \mut^{\Gt}_{i/n}\Big)\bigg]
		\\
		\le ~ &
		\frac{2}{n}\sum_{i = 1}^n \E\bigg[ \Wop^2_2\Big(\mu^G_{i/n}, ~ \mut^{G}_{i/n}\Big)\bigg]
		+\frac{2}{n}\sum_{i = 1}^n \E\bigg[ \Wop^2_2\Big(\mut^G_{i/n}, ~ \mut^{\Gt}_{i/n}\Big)\bigg]
		\\ \le ~ &
		2\int_0^1\E[\Wc_2^2(\mu_v,\mut_v)]\,dv ~ + ~ \frac{2K}{n}\sum_{i = 1}^n\int_0^1|G_{i/n,v} - \Gt_{i/n,v}|\,dv
		\\ \le ~ &
		2\int_0^1\E\bigg[\sup_{t \in [0,T]}|X_{v}(t) - \Xt_{v}(t)|^2\bigg] \,dv ~ + ~ \frac{2K}{n}\sum_{i = 1}^n\int_0^1|G_{i/n,v} - \Gt_{i/n,v}|\,dv.
	\end{align*}
	By the continuity of $G$, we have
	\begin{align*}
		\lim_{n \to \infty}\frac{1}{n}\sum_{i = 1}^n\int_0^1|G_{i/n,v} - \Gt_{i/n,v}|\,dv = \|G - \Gt\|_{L^1} \le \eta.
	\end{align*}
	Finally, the following term remains to be estimated: 
	$$\sup_{v \in [0,1]}\E\Big[\sup_{t \in [0,T]}|X_{v}(t) - \Xt_{v}(t)|^2\Big].$$
	Indeed, there exists some constant $K$, which may vary from line to line, such that
	\begin{align*}
		&
		\sup_{v \in [0,1]}\E\Big[\sup_{t \in [0,T]}|X_{v}(t) - \Xt_{v}(t)|^2\Big]
		\\ \le ~ & 
		K\sup_{v \in [0,1]}\E\bigg[\int_0^T\sum_{\alpha = p,b,w}\bigg|\phi_\alpha\big(X_v(t), \mu^G_{v}(t) \big) - \phi_\alpha\big(\Xt_v(t), \tilde{\mu}^{\tilde{G}}_{v}(t)\big)\bigg|^2 dt\bigg]
		\\ \le  ~ &
		K\sup_{v \in [0,1]}\E\bigg[\int_0^T \sup_{s \in [0,t]}|X_v(s) - \Xt_v(s)|^2 + \Wop^2_2\big(\mu^{G}_{v}(t), ~ \tilde{\mu}^{\tilde{G}}_{v}(t)\big) dt\bigg]
		\\ \le  ~ &
		K\int_0^T \sup_{v \in [0,1]}\E\bigg[\sup_{s \in [0,t]}|X_v(s) - \Xt_v(s)|^2\bigg] dt + K\sup_{v \in [0,1]}\int_0^1|G_{v,u} - \Gt_{v,u}|du
		\\ \le ~ &
		K\sup_{v \in [0,1]}\int_0^1|G_{v,u} - \Gt_{v,u}|du ~ \le ~ K\eta,
	\end{align*}
	where the first inequality holds true by Burkholder-Davis-Gundy inequality, the second one by Lemma \ref{lem: basic coefficient estimate}, the third one by a similar argument as in estimate of \eqref{eq: distance between integral}, the fourth one by Grönwall's inequality, and the last one by that $\|G - \Gt\|_{L^\infty} \le \eta$.
	
	Therefore, we conclude that for any $\eta > 0$, there exists some constant $K$, independent of $n$ and $\eta$, such that
	\begin{align*}
		\limsup_{n \to \infty}\frac{1}{n}\sum_{i = 1}^n \E\Bigg[ \Wop^2_2\bigg(\frac{1}{n}\sum_{j = 1}^n \Gt^n_{i,j} \delta_{\Xt_{{j/n}}}, ~ \tilde{\mu}^{\tilde{G}}_{i/n}(t) \bigg)\Bigg]
		\le K\eta.
	\end{align*}
	This proves the result.
\end{proof}

\medskip

\begin{proof} [Proof of Theorem \ref{thm: weighted mean W2 squared}]
	We first denote for any $i \in [n]$, $n \in \N$ and $t \in [0, T]$
	\begin{align*}
		\bar{\mu}^n_i := \frac{1}{n}\sum_{j = 1}^n G^n_{i,j}\delta_{X^n_j}, \qquad
		\bar{\mu}^n_i(t) := \frac{1}{n} \sum_{j = 1}^n G^n_{i,j}\delta_{X^n_j(t)}.
	\end{align*}

    By Burkholder-Davis-Gundy inequality and Lemma \ref{lem: basic coefficient estimate}, we obtain
	\begin{align*}
		&\E \Big[\sup_{s \in [0,T]} \big| X^n_j(s) - X_{j/n}(s) \big|^2 \Big]
		\\
		& \qquad \qquad \le C\int_0^T \E\bigg[\sup_{s \in [0,t]} \big|X^n_j(s) - X_{j/n}(s)\big|^2\bigg] \, dt +C\E\bigg[\int_0^T\Wop^2_2 \big(\bar{\mu}^n_j(t), \mu^G_{j/n}(t)\big) \, dt \bigg].
	\end{align*}
	Then, by Grönwall's inequality, it follows that
	\begin{align}   \label{ineq: W2 bound}
		\E\bigg[\sup_{t \in [0,T]} \big|X^n_j(t) - X_{j/n}(t) \big|^2 \bigg] ~ \le ~ C \E \bigg[\int_0^T \Wop^2_2 \big(\bar{\mu}^n_j(t), \mu^G_{j/n}(t)\big) \, dt \bigg].
	\end{align}
	On the other hand, we obtain from the triangle inequality for $\Wop$ between the positive measures of $\mathcal{M}_2(\Cc^d)$
	\begin{align*}
		& ~ \frac{1}{n}\sum_{i = 1}^n \E\bigg[\Wop^2_2(\bar{\mu}^n_i,\mu^G_{i/n})\bigg]
		\\
		\le & ~ \frac{2}{n}\sum_{i = 1}^n \E\bigg[ \Wop^2_2\bigg(\frac{1}{n}\sum_{j = 1}^n G^n_{i,j}\delta_{X^n_j}, \frac{1}{n}\sum_{j = 1}^n G^n_{i,j}\delta_{X_{j/n}}\bigg)\bigg]
		+ \frac{2}{n}\sum_{i = 1}^n \E \bigg[\Wop^2_2\bigg( \frac{1}{n} \sum_{j = 1}^n G^n_{i,j} \delta_{X_{j/n}}, \mu^G_{i/n} \bigg) \bigg].
	\end{align*}
	Writing the common total masses of the two measures as
	\begin{equation*}
		\bar{G}^n_i := \frac{1}{n}\sum_{j=1}^n G^n_{i, j} = m_{\frac{1}{n}\sum_{j = 1}^n G^n_{i,j}\delta_{X^n_j}} = m_{\frac{1}{n}\sum_{j = 1}^n G^n_{i,j}\delta_{X_{j/n}}},
	\end{equation*}
	when $\bar{G}^n_i > 0$,
	the term inside the first expectation can be expressed as
	\begin{align*}
		&\Wop^2_2\bigg(\frac{1}{n}\sum_{j = 1}^n G^n_{i,j}\delta_{X^n_j}, \frac{1}{n}\sum_{j = 1}^n G^n_{i,j}\delta_{X_{j/n}}\bigg)
		= \big( \bar{G}^n_i \big)^2 \Wc^2_2 \bigg( \frac{\frac{1}{n}\sum_{j = 1}^n G^n_{i,j}\delta_{X^n_j}}{\bar{G}^n_i}, \frac{\frac{1}{n}\sum_{j = 1}^n G^n_{i,j}\delta_{X_{j/n}}}{\bar{G}^n_i} \bigg)
		\\
		&\le \big(\bar{G}^n_i\big) \frac{1}{n}\sum_{j=1}^n G^n_{i, j} \sup_{t \in [0, T]} \big|X^n_j(t) - X_{j/n}(t)\big|^2
		\le \frac{1}{n}\sum_{j=1}^n G^n_{i, j} \sup_{t \in [0, T]} \big|X^n_j(t) - X_{j/n}(t)\big|^2.
	\end{align*}
	Here, the first equality follows from \eqref{eq: WOP}, the first inequality uses the property of Wasserstein distance between the empirical measures with the coupling $\gamma = \frac{\frac{1}{n}\sum_{j=1}^n G^n_{i, j}\delta_{(x^n_j, X_{j/n})}}{\bar{G}^n_i}$, and the last inequality follows from the bound $\bar{G}^n_{i} \le 1$. 
	When $\bar{G}^n_i = 0$, the inequality still holds since they are all $0$.
	Therefore, we have
	\begin{align}
		& ~ \frac{1}{n}\sum_{i = 1}^n \E\bigg[\Wop^2_2(\bar{\mu}^n_i, \mu^G_{i/n})\bigg]     \label{ineq: WOP before Grönwall}
		\\
		\le & ~ \frac{2}{n^2} \sum_{i,j = 1}^n G^n_{i,j} \E\bigg[\sup_{t \in [0,T]}|X^n_j(t) - X_{j/n}(t)|^2\bigg]
		+\frac{2}{n}\sum_{i = 1}^n \E\bigg[\Wop^2_2\bigg(\frac{1}{n}\sum_{j = 1}^n G^n_{i,j} \delta_{X_{j/n}},\mu^G_{i/n}\bigg)\bigg]          \nonumber
		\\ \le & ~
		\frac{2C}{n^2}\sum_{i,j = 1}^n G^n_{i,j} \E\bigg[\int_0^T\Wop^2_2\big(\bar{\mu}^n_j(t),\mu^G_{j/n}(t)\big)dt\bigg]
		+ \frac{2}{n}\sum_{i = 1}^n \E\bigg[\Wop^2_2\bigg(\frac{1}{n}\sum_{j = 1}^n G^n_{i,j} \delta_{X_{j/n}},\mu^G_{i/n}\bigg)\bigg]          \nonumber
		\\ \le & ~
		2C\int_0^T\frac{1}{n}\sum_{j = 1}^n \E\Big[\Wop^2_2\big(\bar{\mu}^n_j(t),\mu^G_{j/n}(t)\big)\Big]dt
		+\frac{2}{n}\sum_{i = 1}^n \E\bigg[\Wop^2_2\bigg(\frac{1}{n}\sum_{j = 1}^n G^n_{i,j} \delta_{X_{j/n}},\mu^G_{i/n}\bigg)\bigg],          \nonumber
	\end{align}
	where the second inequality uses \eqref{ineq: W2 bound} and the last one follows from the bound $G^n_{i, j} \le 1$.
	
	If we again write the common total masses of the two positive measures as
	\begin{equation*}
		\bar{G}^n_i := \frac{1}{n}\sum_{j=1}^n G^n_{i, j} = m_{\bar{\mu}^n_i} = m_{\bar{\mu}^n_i(t)}, \qquad \bar{G}_{{i/n}} := m_{\mu^G_{i/n}} = m_{\mu^G_{i/n}(t)}, \qquad \forall \, t \in [0, T],
	\end{equation*}
	respectively, and the probability measures
	\begin{equation*}
		\bar{\mu}^{\circ n}_i := \frac{\bar{\mu}^n_i}{\bar{G}^n_i}, \qquad \bar{\mu}^{\circ n}_i(t) := \frac{\bar{\mu}^n_i(t)}{\bar{G}^n_i},
        \qquad \bar{\mu}^\circ_{{i/n}} = \frac{\mu^G_{i/n}}{\bar{G}^{{i/n}}},
        \qquad \bar{\mu}^\circ_{{i/n}}(t) = \frac{\mu^G_{i/n}(t)}{\bar{G}^{{i/n}}}, \qquad \forall \, i \in [n], ~n \in \N,
	\end{equation*}
	the definition of $\Wop_2$ gives for every $t \in [0, T]$
	\begin{align}
		\Wop^2_2 \big(\bar{\mu}^n_i(t), \mu^G_{i/n}(t) \big) &= \big(\bar{G}^n_i - \bar{G}_{{i/n}} \big)^2 + \Wc^2_2 \Big( T_{\bar{G}^n_i}\# \bar{\mu}^{\circ n}_i(t), ~ T_{\bar{G}_{{i/n}}}\#\bar{\mu}^\circ_{{i/n}}(t)\Big)   \label{ineq: two WOPs}
		\\
		&\le \big(\bar{G}^n_i - \bar{G}_{{i/n}} \big)^2 + \Wc^2_2 \Big( T_{\bar{G}^n_i}\# \bar{\mu}^{\circ n}_i, ~ T_{\bar{G}_{{i/n}}}\#\bar{\mu}^\circ_{{i/n}}\Big)
		= \Wop^2_2 \big(\bar{\mu}^n_i, \mu^G_{i/n} \big),       \nonumber
	\end{align}
	where the notation $T_a\# \nu$ is the pushforward of $\nu$ by $T_a(x) = ax$ (taking the reference point $x_0 = 0$) for $\nu$ either in $\mathcal{M}_2(\R^d)$ or in $\mathcal{M}_2(\mathcal{C}^d)$. Here, the above inequality follows from
	\begin{align*}
		\sup_{t \in [0, T]}\Wc^2_2 \big(\mu(t), \nu(t) \big) \le \Wc^2_2(\mu, \nu)
	\end{align*}
	for any $\mu, \nu \in \mathcal{P}_2(\mathcal{C}^d)$, since we have a series of inequalities
	\begin{align*}
		\Wc_2^2 \big(\mu(t), \nu(t)\big)
		\le
		\E[|X(t) - Y(t)|^2]
		\le
		\E[\sup_{t \in [0, T]} |X(t) - Y(t)|^2] = \E \Vert X-Y \Vert^2
	\end{align*}
	for arbitrary $\Cc^d$-valued random variables $X, Y$ with $\mu = \Lc(X)$, $\nu = \Lc(Y)$ with $\mu(t) = \Lc(X(t))$, $\nu(t) = \Lc(Y(t))$ for any $t \in [0, T]$, along with the definition of 2-Wasserstein distance and the arbitrariness of $X$, $Y$, and $t$.
	
	Plugging \eqref{ineq: two WOPs} into \eqref{ineq: WOP before Grönwall} and using Grönwall's inequality, we have
	\begin{align*}
		~ \frac{1}{n}\sum_{i = 1}^n
		\E\bigg[\Wop^2_2(\bar{\mu}^n_i,\mu^G_{i/n})\bigg]
		\le  ~
		\frac{C}{n}\sum_{i = 1}^n \E \bigg[\Wop^2_2\bigg(\frac{1}{n}\sum_{j = 1}^n G^n_{i,j} \delta_{X_{j/n}},\mu^G_{i/n}\bigg)\bigg],
	\end{align*}
	where the right-hand side tends to $0$ as $n$ goes to $\infty$ from Theorem \ref{thm: relaxed_weighted mean W2 continuity}. This proves \eqref{conv: weighted mean W2 squared}.
\end{proof}

\medskip

\begin{proof} [Proof of Proposition \ref{thm: W2 to zero}]
	Using the triangle inequality, the property of Wasserstein distance between two empirical measures, and the bound \eqref{ineq: W2 bound}, we derive
	\begin{align*}
		\E \big[\Wc^2_2(\mu^n,\bar{\mu}) \big]
		\le & ~ 2 \E \big[\Wc^2_2(\mu^n, \bar{\mu}^n)\big]
		+ 2 \E\big[\Wc^2_2(\bar{\mu}^n, \bar{\mu})\big]
		\\
		\le & ~ \frac{2}{n} \sum_{i = 1}^n \E\bigg[\sup_{t \in [0,T]} \big|X^n_i(t) - X_{i/n}(t) \big|^2 \bigg]
		+ 2 \E\big[\Wc^2_2(\bar{\mu}^n, \bar{\mu})\big]
		\\
		\le & ~ \frac{C}{n}\sum_{i = 1}^n \E\bigg[\int_0^T \Wop^2_2 \big(\bar{\mu}^n_i(t), \mu^G_{i/n}(t) \big) dt \bigg] + 2 \E \big[\Wc^2_2(\bar{\mu}^n, \bar{\mu})\big]
		\\
		\le & ~ C\int_0^T\frac{1}{n}\sum_{j = 1}^n\E\Big[\Wop^2_2 \big(\bar{\mu}^n_j(t), \mu^G_{j/n}(t) \big) \Big] dt
		+ 2\E[\Wc^2_2(\bar{\mu}^n, \bar{\mu})].
	\end{align*}
	Thanks to Proposition \ref{thm: bar mu_n to bar mu} and Theorem \ref{thm: weighted mean W2 squared}, the last term tends to $0$, as $n$ goes to $\infty$, which proves \eqref{conv: mean W2 squared}.
\end{proof}

\medskip

\begin{proof} [Proof of Theorem \ref{thm: relaxed_weighted mean W2 squared}]
	The proof is similar to that of Theorem \ref{thm: relaxed_weighted mean W2 continuity}. For a given measurable graphon $G$ and any constant $\eta > 0$, there exists a continuous graphon $\Gt : [0,1]^2 \longrightarrow [0,1]$ with
	\begin{align*}
		\|G - \Gt\|_{L^1} \le \eta,
	\end{align*}
	and a sequence of discrete graphons $\Gt^n: [n] \x [n] \longrightarrow [0,1]$ for each $n \in \N$ such that
	\begin{align*}
		\sup_{n \in \N}\|G^n - \Gt^n\|_{L^1} \le \eta \qquad \text{and} \qquad
		\quad \lim_{n \to \infty} \|\Gt^n - \Gt\|_\square ~ = ~ 0.
	\end{align*}
	Let $\{\Xt_v\}_{v\in [0,1]}$ be the solution to the graphon system \eqref{eq : graphon particle system with common noise} with $\Gt$, 
	$\{\Xt^n_i\}_{i\in [n]}$ be the solution to the finite particle system \eqref{eq : finite particle system with common noise} with $\Gt^n$, and define $\mut_v := \Lc(\Xt_v|W)$. 
	By Proposition \ref{thm: W2 to zero}, we have
	\begin{align*}
		\E\Bigg[ \Wc^2_2\bigg(\frac{1}{n}\sum_{j = 1}^n\delta_{\Xt^n_j}, \int_{0}^1\mut_vdv \bigg) \Bigg] \xlongrightarrow{n \to \infty} 0.
	\end{align*}
	
	We now try to estimate the two terms
	\begin{align*}
		\E\Bigg[ \Wc^2_2\bigg(\frac{1}{n}\sum_{j = 1}^n\delta_{X^n_j},\frac{1}{n}\sum_{j = 1}^n\delta_{\Xt^n_j} \bigg) \Bigg],
		\qquad
		\E\Bigg[ \Wc^2_2\bigg(\int_{0}^1\mu_vdv, \int_{0}^1\mut_vdv \bigg) \Bigg].
	\end{align*}
    For the first term, by the property of Wasserstein distance between empirical measures, Burkholder-Davis-Gundy inequality, Lemma \ref{lem: basic coefficient estimate}, and Grönwall's inequality, 
	\begin{align*}
		& \E\Bigg[ \Wc^2_2\bigg(\frac{1}{n}\sum_{j = 1}^n\delta_{X^n_j},\frac{1}{n}\sum_{j = 1}^n\delta_{\Xt^n_j} \bigg) \Bigg]
		\le \frac{1}{n}\sum_{j = 1}^n \E\bigg[\sup_{t \in [0,T]}|X^n_j(t) - \Xt^n_j(t)|^2\bigg]
		\\ \le ~ &
		\frac{K}{n}\sum_{j = 1}^n\E\bigg[\int_0^T \sup_{s \in [0,t]}|X^n_j(s) - \Xt^n_j(s)|^2 + \Wop_2^2\bigg(\frac{1}{n}\sum_{k = 1}^n G^n_{j,k}\delta_{X^n_k(t)},
        \frac{1}{n}\sum_{k = 1}^n \Gt^n_{j,k}\delta_{\Xt^n_k(t)}\bigg)dt\bigg]
		\\ \le ~ &
		K\int_0^T \frac{1}{n}\sum_{j = 1}^n \E\bigg[\sup_{s \in [0,t]}|X^n_j(s) - \Xt^n_j(s)|^2\bigg] dt ~+~ 2K\|G^n - \Gt^n\|_{L^1}
		\le K\|G^n - \Gt^n\|_{L^1} \le K\eta.
	\end{align*}
	For the second term, we have a similar estimate that
	\begin{align*}
		\E\Bigg[ \Wc^2_2\bigg(\int_{0}^1\mu_vdv, \int_{0}^1\mut_vdv \bigg) \Bigg] \le \int_0^1\E[\Wc^2_2(\mut_v,\mu_v)]dv \le K\|G - \Gt\|_{L^1} \le K\eta.
	\end{align*}
	Therefore, for any $\eta > 0$, there exists some constant $K$ independent of $n$ and $\eta$ such that
	\begin{align*}
		\limsup_{n \to \infty} \E\Bigg[ \Wc^2_2\bigg(\frac{1}{n}\sum_{j = 1}^n\delta_{X^n_j}, \int_{0}^1\mu_vdv \bigg) \Bigg]
		\le K\eta.
	\end{align*}
\end{proof}

\medskip

\begin{proof} [Proof of Lemma \ref{lem:sampling_rate}]
	\textbf{Part 1: Reduction to the probability measure}
	
	We define $m_{i,n} := \frac{1}{n}\sum_{j = 1}^nG^n_{i,j}$, then
	by (ii) of Lemma \ref{Lem: WOP R^d}, one has that for $m_{i,n} > 0$,
	\begin{align*}
		\E\Bigg[\Wop^2_2\bigg(\frac{1}{n}\sum_{j = 1}^nG^n_{i,j}\delta_{X_{{j/n}}(t)}, 
		&\frac{1}{n}\sum_{j = 1}^nG^n_{i,j} \mu_{j/n}(t) \bigg)\Bigg]
		\\
		~ &\le ~ m^2_{i, n}
		\E\Bigg[\Wc^2_2\bigg(\frac{1}{n}\sum_{j = 1}^n\frac{G^n_{i,j}}{m_{i,n}} \delta_{X_{{j/n}}(t)}, 
		\frac{1}{n}\sum_{j = 1}^n\frac{G^n_{i,j}}{m_{i, n}} \mu_{j/n}(t) \bigg)\Bigg],
	\end{align*}
	and for $m_{i, n} = 0$,
	\begin{align*}
		\E\Bigg[\Wop^2_2\bigg(\frac{1}{n}\sum_{j = 1}^nG^n_{i,j}\delta_{X_{{j/n}}(t)}, 
		\frac{1}{n}\sum_{j = 1}^nG^n_{i,j} \mu_{j/n}(t) \bigg)\Bigg]
		~ = ~ 0,
	\end{align*}
	which clearly satisfies the desired inequality \eqref{eq:rate_marginal}.
	Therefore, we only focus on the case $m_{i, n} > 0$ from now on, and for simplicity, we denote
	\begin{equation*}
		\delta^i_n(t) := \frac{1}{n}\sum_{j = 1}^n\frac{G^n_{i,j}}{m_{i, n}}\delta_{X_{{j/n}}(t)} \in \mathcal{P}(\R^d),
		\qquad
		\mu^i_n(t) := \frac{1}{n}\sum_{j = 1}^n\frac{G^n_{i,j}}{m_{i, n}} \mu_{j/n}(t) \in \mathcal{P}(\R^d).
	\end{equation*}
	
	\noindent \textbf{Part 2: Key estimates}
	
	By Lemmas 5 and 6 in \cite{FournierGuillin}, one has that for any $\mu,\nu \in \Pc_2(\R^d)$
	\begin{equation}\label{eq:child_estimate}
		\Wc^2_2(\mu,\nu)
		~ \le ~
		K_d \sum_{k = 0}^\infty 2^{2k}\sum_{l = 0}^\infty 2^{-2l}\sum_{F \in \Pc_l}|\mu(2^kF \cap B_k) - \nu(2^kF \cap B_k)|,
	\end{equation}
	where the constant $K_d$ depends only on dimension $d$, the subsets 
	\[
		B_0 := (-1,1]^d, \qquad B_k :=  (-2^k,2^k]^d\backslash (-2^{k-1},2^{k-1}]^d
	\]
	for $k \in \N$ constitute a partition of $\R^d$, $\Pc_l$ denotes the natural partition of $B_0$ into $2^{dl}$ translations of $(-2^{-l},2^{-l}]^d$ for $l \in \N$, and $2^kF := \{2^k x : x \in F\}$ for $F \subset \R^d$.
	
	Then, to estimate $\sum_{F \in \Pc_l}|\delta^i_n(t)(2^kF \cap B_k) - \mu^i_n(t)(2^kF \cap B_k)|$ for each $n$, we claim that for each $A \in \mathcal{B}(\R^d)$
	\begin{align}\label{eq:each_child_estimate}
		\E[|\delta^i_n(t)(A) - \mu^i_n(t)(A)|]
		~ \le ~
		\min\Bigg\{2\E[\mu^i_n(t)(A)], \sqrt{\frac{\E[\mu^i_n(t)(A)]}{nm_{i,n}}} \Bigg\}.
	\end{align}
	and that for some positive constant $C$ independent of $n$ and $i$, 
	\begin{align}\label{eq:partition_estimate}
		m_{i,n}\E[\mu^i_n(t)(B_k)]
		~ \le ~
		C2^{-(2 + \eps)(k - 1)}.
	\end{align}        
	In fact, we have 
	\begin{align*}
		\E[|\delta^i_n(t)(A) - \mu^i_n(t)(A)|] 
		~ \le & ~
		\frac{1}{n m_{i,n}}\sum^n_{j = 1}G^n_{i,j}
		\E\bigg[\Big|\delta_{X_{{j/n}}(t)}(A) -  \mu_{j/n}(t)(A)\Big|\bigg]
		\\ ~ \le & ~
		\frac{1}{n m_{i,n}}\sum^n_{j = 1}G^n_{i,j}
		\E\bigg[\delta_{X_{{j/n}}(t)}(A) +  \mu_{j/n}(t)(A)\bigg]
		~ = ~ 2\E[\mu^i_n(t)(A)],
	\end{align*}
	and
	\begin{align*}
		& \Big(\E\big[|\delta^i_n(t)(A) - \mu^i_n(t)(A)|\big]\Big)^2
		~ \le ~
		\E \big[|\delta^i_n(t)(A) - \mu^i_n(t)(A)|^2 \big]
		\\ ~ = & ~
		\frac{1}{n^2 m^2_{i,n}}\sum^n_{j,k = 1}G^n_{i,j}G^n_{i,k}
		\E\bigg[\Big(\delta_{X_{{j/n}}(t)}(A) -  \mu_{j/n}(t)(A)\Big)
		\Big(\delta_{X_{k/n}(t)}(A) -  \mu_{j/n}(t)(A)\Big)\bigg]
		\\ ~ = & ~
		\frac{1}{n^2 m^2_{i,n}}\sum^n_{j = 1} (G^n_{i,j})^2
		\E\bigg[\Big(\delta_{X_{{j/n}}(t)}(A) -  \mu_{j/n}(t)(A)\Big)^2 \bigg]
		\\ ~ = & ~
		\frac{1}{n^2 m^2_{i,n}}\sum^n_{j = 1} (G^n_{i,j})^2
		\E\bigg[\delta_{X_{{j/n}}(t)}(A) + \Big( \mu_{j/n}(t)(A)\Big)^2 - 2\delta_{X_{{j/n}}(t)}(A) \mu_{j/n}(t)(A) \bigg]
		\\ ~ = & ~
		\frac{1}{n^2 m^2_{i,n}}\sum^n_{j = 1} (G^n_{i,j})^2
		\E\bigg[\delta_{X_{{j/n}}(t)}(A) - \Big( \mu_{j/n}(t)(A)\Big)^2
		\bigg]
		\\ ~ \le & ~
		\frac{1}{n^2 m^2_{i,n}}\sum^n_{j = 1} G^n_{i,j} \E\bigg[\delta_{X_{{j/n}}(t)}(A)\bigg]
		~ = ~ \frac{\E[\mu^i_n(t)(A)]}{nm_{i,n}},
	\end{align*}
	where the second and fourth equalities hold by the conditional independence of $X_{{j/n}}$ and $X_{\ell/n}$ for $j \neq \ell$, and the tower property of expectation, i.e., for $j\neq \ell$
	\begin{align*}
		&~\E\bigg[\bigg(\delta_{X_{{j/n}}(t)}(A) - \mu_{j/n}(t)(A)\bigg)
		\bigg(\delta_{X_{\ell/n}(t)}(A) - \mu_{\ell/n}(t)(A)\bigg)\bigg]
		\\
		= &~
		\E\bigg[\Big( \delta_{X_{{j/n}}(t)} \delta_{X_{\ell/n}(t)} - \delta_{X_{{j/n}}(t)}\mu_{\ell/n}(t) 
		- \mu_{j/n}(t) \delta_{X_{{\ell/n}}(t)} + \mu_{j/n}(t)\mu_{\ell/n}(t)\Big) (A)\bigg]
		\\
		= &~
		\E\bigg[\E\Big[\delta_{X_{{j/n}}(t)}(A)\delta_{X_{\ell/n}(t)}(A)|W\Big] - \E\Big[\delta_{X_{{j/n}}(t)}(A)|W\Big] \mu_{\ell/n}(t)(A) 
		\\ & \qquad \qquad \qquad \qquad \qquad
		- \mu_{j/n}(t)(A) \E\Big[\delta_{X_{{\ell/n}}(t)}(A)|W\Big] + \mu_{j/n}(t)(A)\mu_{\ell/n}(t)(A)\bigg]
		\\
		= & ~
		\E\bigg[ \E\Big[\delta_{X_{{j/n}}(t)}(A)|W\Big] \E\Big[\delta_{X_{\ell/n}(t)}(A)|W\Big] - \mu_{j/n}(t)(A) \mu_{\ell/n}(t)(A)
		\\ & \qquad \qquad \qquad \qquad \qquad - \mu_{j/n}(t)(A)\mu_{\ell/n}(t)(A) + \mu_{j/n}(t)(A)\mu_{\ell/n}(t)(A)\bigg] = 0.
	\end{align*}
	Finally, we can prove \eqref{eq:partition_estimate}, if we observe that
	\begin{align*}
		m_{i,n}\E[\mu^i_n(t)(B_k)]
		\le & ~ m_{i,n}\E\bigg[\int_{B_k}\bigg(\frac{|x|}{2^{k - 1}}\bigg)^{2 + \eps}\mu^i_n(t)(dx)\bigg]
		\le m_{i,n}\E\bigg[\int_{\R^d}\bigg(\frac{|x|}{2^{k - 1}}\bigg)^{2 + \eps}\mu^i_n(t)(dx)\bigg]
		\\ ~ = & ~
		2^{-(2 + \eps)(k - 1)}
		\frac{1}{n}\sum_{j = 1}^nG^n_{i,j}\E[|X_{{j/n}}(t)|^{2 + \eps}]
		\\ ~ \le & ~
		2^{-(2 + \eps)(k - 1)}\sup_{u \in [0,1]}\E\bigg[\sup_{t \in [0,T]}|X_{u}(t)|^{2 + \eps}\bigg]
		~ \le ~ C2^{-(2 + \eps)(k - 1)},
	\end{align*}
	since $|x| \ge 2^{k-1}$ in $B_k$, and the last inequality holds by Lemma \ref{lem: unique pathwise solution}.
	
	\noindent \textbf{Part 3: Application of the key estimates}
	
	With the estimates \eqref{eq:each_child_estimate} and \eqref{eq:partition_estimate}, we observe that for $F_1, F_2 \in \Pc_l$, with $F_1 \neq F_2$, $F_1 \cap F_2 = \emptyset$:
	\begin{align*}
		&
		\sum_{F \in \Pc_l}\E[|\delta^i_n(t)(2^kF \cap B_k) - \mu^i_n(t)(2^kF \cap B_k)|]
		\\ ~ \le & ~
		\sum_{F \in \Pc_l}\min\Bigg\{2\E[\mu^i_n(t)(2^kF \cap B_k)], \sqrt{\frac{\E[\mu^i_n(t)(2^kF \cap B_k)]}{n m_{i,n}}} \Bigg\}
		\\ ~ \le& ~
		\min\Bigg\{2\sum_{F \in \Pc_l}\E[\mu^i_n(t)(2^kF \cap B_k)], \sum_{F \in \Pc_l}\sqrt{\frac{\E[\mu^i_n(t)(2^kF \cap B_k)]}{n m_{i,n}}} \Bigg\}
		\\ ~ \le& ~
		\min\Bigg\{2\E[\mu^i_n(t)(B_k)], 2^{\frac{dl}{2}}\sqrt{\frac{\E[\mu^i_n(t)(B_k)]}{n m_{i,n}}} \Bigg\}
		\le \frac{C}{m_{i,n}}\min\Bigg\{2^{-(2 + \eps)(k - 1) +1}, 2^{\frac{dl}{2}}\sqrt{\frac{2^{-(2 + \eps)(k - 1)}}{n}} \Bigg\}
		\\ ~ \le & ~
		\frac{C2^{3+\eps}}{m_{i,n}}\min\Bigg\{2^{-(2 + \eps)k}, 2^{\frac{dl}{2}}\sqrt{\frac{2^{-(2 + \eps)k}}{n}} \Bigg\}
	\end{align*}
	where the third inequality uses Cauchy-Schwarz inequality and the fact that $\#(\Pc_l) = 2^{dl}$.
	From all the estimates above, we have for some constant $C_d > 0$,
	\begin{align*}
		&
		\E\Bigg[\Wop^2_2\bigg(\frac{1}{n}\sum_{j = 1}^nG^n_{i,j}\delta_{X_{{j/n}}(t)}, 
		\frac{1}{n}\sum_{j = 1}^nG^n_{i,j} \mu_{j/n}(t) \bigg)\Bigg]
		\\~ \le &~ m_{i,n}^2
		\E\Bigg[\Wc^2_2\bigg(\frac{1}{n}\sum_{j = 1}^n\frac{G^n_{i,j}}{m_{i,n}}\delta_{X_{{j/n}}(t)}, 
		\frac{1}{n}\sum_{j = 1}^n\frac{G^n_{i,j}}{m_{i,n}} \mu_{j/n}(t) \bigg)\Bigg]
		\\ ~ \le & ~
		m_{i,n}^2K_d \sum_{k = 0}^\infty 2^{2k}\sum_{l = 0}^\infty 2^{-2l}\frac{C2^{3+\eps}}{m_{i,n}}\min\Bigg\{2^{-(2 + \eps)k}, 2^{\frac{dl}{2}}\sqrt{\frac{2^{-(2 + \eps)k}}{n}} \Bigg\}
		\\ ~ \le & ~
		C_d2^{3+\eps} \sum_{k = 0}^\infty 2^{2k}\sum_{l = 0}^\infty 2^{-2l}\min\Bigg\{2^{-(2 + \eps)k}, 2^{\frac{dl}{2}}\sqrt{\frac{2^{-(2 + \eps)k}}{n}} \Bigg\},
	\end{align*}
	where the last inequality holds by $m_n \le 1$. The rest of the estimate follows exactly by the same argument as in the proof of Theorem 1 in \cite{FournierGuillin} with $q = 2 + \eps$, thus we conclude the proof.
\end{proof}

\medskip

\begin{proof} [Proof of Theorem \ref{thm: convergence rates}]
	From \eqref{ineq: W2 bound} in the proof of Theorem \ref{thm: weighted mean W2 squared}, we have that for some constant $K > 0$, which may vary from line to line throughout the proof
	\begin{align*}
		\E\bigg[\sup_{t \in [0,T]} \big|X^n_i(t) - X_{i/n}(t) \big|^2 \bigg] 
		~ \le &~ 
		K \E \bigg[\int_0^T \Wop^2_2 \bigg(\frac{1}{n} \sum_{j = 1}^n G^n_{i,j}\delta_{X^n_j(t)}, 
		\mu^G_{i/n}(t) \bigg) \, dt \bigg]
		\\ ~ \le & ~
		K\E \bigg[\int_0^T \Wop^2_2 \bigg(\frac{1}{n} \sum_{j = 1}^n G^n_{i,j}\delta_{X^n_j(t)}, 
		\frac{1}{n} \sum_{j = 1}^n G^n_{i,j}\mu_{j/n}(t)\bigg) \, dt \bigg]
		\\ &+
		K\E \bigg[\int_0^T \Wop^2_2 
		\bigg(\frac{1}{n} \sum_{j = 1}^n G^n_{i,j}\mu_{j/n}(t), \mu^G_{i/n}(t) \bigg) \, dt \bigg].
	\end{align*}
	For the first term on the right-hand side, by Lemma \ref{lem:sampling_rate}, there exists some positive constant $K_{d,\eps} > 0$, which may vary from line to line,
	\begin{align*}
		&\E \bigg[\Wop^2_2 \bigg(\frac{1}{n} \sum_{j = 1}^n G^n_{i,j}\delta_{X^n_j(t)}, \frac{1}{n} \sum_{j = 1}^n G^n_{i,j}\mu_{j/n}(t)\bigg)\bigg]
		\\~ \le &~
		2\E \bigg[\Wop^2_2 \bigg(\frac{1}{n} \sum_{j = 1}^n G^n_{i,j}\delta_{X^n_j(t)}, \frac{1}{n} \sum_{j = 1}^n G^n_{i,j}\delta_{X_{j/n}(t)}\bigg) \bigg]
		\\ & \qquad \qquad \qquad  + 
		2\E \bigg[\Wop^2_2 \bigg(\frac{1}{n} \sum_{j = 1}^n G^n_{i,j}\delta_{X_{j/n}(t)}, \frac{1}{n} \sum_{j = 1}^n G^n_{i,j}\mu_{j/n}(t)\bigg)\bigg]
		\\ ~ \le & ~
		\frac{2}{n}\sum_{j = 1}^n G^n_{i,j}\E\bigg[\big|X^n_j(t) - X_{j/n}(t) \big|^2 \bigg] + K_{d,\eps}M_n
		\le 2\sup_{i \in [n]}\E\bigg[\sup_{s \in [0,t]}|X^n_i(s) - X_{{i/n}}(s)|^2\bigg] + K_{d,\eps}M_n.
	\end{align*}
	By Grönwall's inequality, it follows that
	\begin{align*}
		\sup_{i \in [n]}\E\bigg[\sup_{t \in [0,T]}|X^n_i(t) - X_{{i/n}}(t)|^2\bigg]
		\le K\E \bigg[\int_0^T \Wop^2_2 
		\bigg(\frac{1}{n} \sum_{j = 1}^n G^n_{i,j}\mu_{j/n}(t), \mu^G_{i/n}(t) \bigg) \, dt \bigg] + K_{d,\eps}M_n.
	\end{align*}
	Then, it remains to compute the first term on the right-hand side of the above inequality. 
	Let us define 
	\begin{align*}
		\Gt^n_{u,v} := \sum_{i,j = 1}^nG^n_{i,j}\mathds{1}_{(\frac{i - 1}{n},\frac{i}{n}] \x (\frac{j - 1}{n},\frac{j}{n}]}(u,v), \qquad
		~\mbox{for}~(u,v) \in [0,1]^2
		\\
		\mut^n_{v}(t) := \sum_{j = 1}^n\mu_{j/n}(t)\mathds{1}_{(\frac{j - 1}{n},\frac{j}{n}]}(v), \qquad
		~\mbox{for}~v \in [0,1], \quad t \in [0,T],
	\end{align*}
	then, by the triangle inequality, 
	\begin{align*}
		\E \bigg[\int_0^T \Wop^2_2 
		\bigg(\frac{1}{n} \sum_{j = 1}^n G^n_{i,j}\mu_{j/n}(t), \mu^G_{i/n}(t)\bigg) \, dt \bigg]
		~ \le ~
		2\Dc^n_i + 2\Gc^n_i,
	\end{align*}
	where for $n \in \N$, $i \in [n]$, $t \in [0,T]$, $\Dc^n_i$ and $\Gc^n_i$ are given by
	\begin{align*}
		\Dc^n_i
		:=
		\E \bigg[\int_0^T\Wop^2_2 
		\bigg(\frac{1}{n} \sum_{j = 1}^n G^n_{i,j}\mu_{j/n}(t), \mu^{\Gt}_{i/n}(t) \bigg) dt \bigg],
		\quad          
		\Gc^n_i
		:=
		\E \bigg[\int_0^T \Wop^2_2 
		\big(\mu^{\Gt}_{i/n}(t), \mu^G_{i/n}(t) \big) dt \bigg].
	\end{align*}
	As for $\Dc^n_i$, we observe that
	if $m_{i,n} := \frac{1}{n}\sum_{i = 1}^nG^n_{i,j} = 0$, then $\Dc^n_i = 0$. When $m_{i,n} > 0$, it follow that
	\begin{align*}
		\Dc^n_i
		~ = &~
		\E\bigg[\int_0^T\Wop^2_2 
		\big( \mut^{\Gt}_{i/n}(t), \mu^{\Gt}_{i/n}(t) \big)\, dt\bigg]
		\\~ \le & ~
		m^2_{i,n}
		\E\bigg[\int_0^T\Wc^2_2 
		\bigg( \int_{0}^1 \frac{\Gt^n_{i/n,v}}{m_{i,n}}\mut^n_v(t)dv, 
		\int_{0}^1 \frac{\Gt^n_{i/n,v}}{m_{i,n}}\mu_v(t)dv\bigg)\, dt\bigg]
		\\ ~ \le & ~
		m^2_{i,n}
		\int_0^T\int_{0}^1 \frac{\Gt^n_{i/n,v}}{m_{i,n}}
		\E\big[\Wc_2^2\big(\mut^n_v(t), 
		\mu_v(t)\big)\big]dv\, dt
		\le \int_0^T\int_{0}^1 \E\big[\Wc_2^2\big(\mut^n_v(t), \mu_v(t)\big)\big]dv\, dt \le \frac{K}{n},
	\end{align*}
	where the first inequality holds by $(ii)$ of Lemma \ref{Lem: WOP R^d}, the second one by the continuity estimate from Kantorovitch duality of Wasserstein distance as in the proof of Lemma 3.1 in \cite{nonlinear:graphon},
	and the last one by \eqref{ineq: W^2_2} with $\eps = 0$ for some positive constant $K$.
	
	As for the term $\Gc^n_i$, when $m_{i,n} = 0$ or $m_{i/n} := \int_0^1 G_{i/n,v}dv = 0$, there exists some positive constant $K$, which may vary from line to line, such that
	\begin{align*}
		\Gc^n_i = ~&\E \bigg[\int_0^T 
		(m_{i,n} - m_{i/n})^2 + 
		(m_{i,n} - m_{i/n})    
		\int_{0}^1  \big(\Gt^n_{i/n,v} - G_{i/n,v}\big)\int_{\R^d}|x|^2\mu_v(t)(dx)dv
		\, dt \bigg]
		\\ \le ~& 
		T\int_0^1 \big(\Gt^n_{i/n,v} - G_{i/n,v}\big)^2 dv
		+ T(m_{i,n} - m_{i/n})\int_0^1 \big(\Gt^n_{i/n,v} - G_{i/n,v}\big)\E\bigg[\sup_{t \in [0,T]}|X_v(t)|^2\bigg] dv
		\\ \le ~ &
		K\int_0^1 \big(\Gt^n_{i/n,v} - G_{i/n,v}\big)^2 dv
		~ \le ~ \frac{K}{n^2}.
	\end{align*}
	Then, for the case $m_{i,n}, m_{i/n} > 0$, by Theorem 6.15 in \cite{villani2016optimal}, there exists some positive constant $K$, which may vary from line to line, satisfying
	\begin{align*}
		\Gc^n_i \le ~&  \frac{K}{n^2} + \E \bigg[\int_0^T 
		2m_{i,n}m_{i/n}\int_{\R^d}|x|^2\bigg|\int_{0}^1 \frac{\Gt^n_{i/n,v}}{m_{i,n}}\mu_v(t)dv - 
		\int_{0}^1 \frac{G_{i/n,v}}{m_{i/n}}\mu_v(t)dv\bigg|(dx)
		\, dt \bigg]
		\\ \le ~ &
		\frac{K}{n^2} + 
		\E \bigg[\int_0^T 
		2m_{i,n}m_{i/n}\int_{0}^{1}\bigg|\frac{\Gt^n_{i/n,v}}{m_{i,n}}- \frac{G_{i/n,v}}{m_{i/n}} \bigg|\int_{\R^d}|x|^2\mu_v(t)(dx)\,dv\, dt\bigg]
		\\ \le ~ &
		\frac{K}{n^2} + 
		K\E \bigg[\int_0^T 
		\int_{0}^{1}\Big|m_{i/n}\Gt^n_{i/n,v}- m_{i/n}G_{i/n,v} + m_{i/n}G_{i/n,v} -
		m_{i,n}G_{i/n,v}\Big|\,dv\,dt\bigg]
		\\ \le ~ &
		\frac{K}{n^2} + K\int_0^T 
		\int_{0}^{1}\big|\Gt^n_{i/n,v}- G_{i/n,v}\big| + \big|m_{i/n} - m_{i,n}\big|\,dv\,dt
		\le \frac{K}{n^2} + \frac{K}{n}.
	\end{align*}
	Finally, we can conclude that there exists some positive constant $K_{d,\eps}$ satisfying
	\begin{align*}
		\sup_{i \in [n]}\E\bigg[\sup_{t \in [0,T]}|X^n_i(t) - X_{{i/n}}(t)|^2\bigg]
		\le 
		\frac{K}{n^2} + \frac{K}{n}+ K_{d,\eps}M_n
		~ \le ~
		K_{d,\eps}M_n.
	\end{align*}
\end{proof}

\bigskip

\section{Characterization of $\Pc(\Mc_{2,K}(\Cc^d) \x \Mc_{2,K}(\Cc^d))$}

The following two characterization results are inspired by Proposition A.3 in \cite{DTP}. The aim is to find a more feasible way to identify that two probability measures on $\Mc_{2,K}(\Cc^d) \x \Mc_{2,K}(\Cc^d)$ are the same.
\begin{thm}\label{thm: characterization}
	Suppose that $\Gamma_1, \Gamma_2 \in \Pc(\Mc_{2,K}(\Cc^d) \x \Mc_{2,K}(\Cc^d))$ satisfy 
	\begin{align}\label{eq:characterization}
		\int_{\Mc^2_{2,K}(\Cc^d)} \prod_{i = 1}^k\prod_{j = 1}^l \langle\varphi_i,\mu\rangle \langle\phi_j,\nu\rangle \Gamma_1(d\mu,d\nu)
		= 
		\int_{\Mc^2_{2,K}(\Cc^d)} \prod_{i = 1}^k\prod_{j = 1}^l \langle\varphi_i,\mu\rangle \langle\phi_j,\nu\rangle \Gamma_2(d\mu,d\nu)
	\end{align}
	for all $k, l \ge 1$, and $\varphi_i, \phi_j \in C_b(\Cc^d)$ for $i \in [k], j \in [l]$. Then, $\Gamma_1 = \Gamma_2$.
\end{thm}

\begin{proof}
	First, by linearity of integral and arbitrariness of $k$, we can generalize \eqref{eq:characterization} as follows:
	for any polynomials $\Psi_i, \Phi_j$ on $\R$ and $i \in [k]$, $j \in [l]$
	\begin{align}
		&~~~\int_{\Mc^2_K(\Cc^d)}
		\prod_{i = 1}^k\prod_{j = 1}^l
		\Psi_i(\langle\varphi_i,\mu\rangle)\Phi_j( \langle\phi_j,\nu\rangle)
		\Gamma_1(d\mu,d\nu)		\label{eq:polynomial_characterization}
		\\
		&= 
		\int_{\Mc^2_K(\Cc^d)}
		\prod_{i = 1}^k\prod_{j = 1}^l
		\Psi_i(\langle\varphi_i,\mu\rangle)\Phi_j( \langle\phi_j,\nu\rangle)
		\Gamma_2(d\mu,d\nu).	\nonumber
	\end{align}
	Since every bounded continuous function on $\R$ can be approximated by polynomials, we can further prove \eqref{eq:polynomial_characterization} for any bounded continuous functions $\Psi_i, \Phi_j$ on $\R$. 
	Again, notice that for any $r \in \R$, the function $\mathds{1}_{(-\infty,r)}$ can be approximated by a sequence of bounded continuous functions, we have for any $r_i, s_j \in \R$, $i \in [k]$, $j \in [l]$,
	\begin{align*}
		\int_{\Mc^2_K(\Cc^d)}
		\prod_{i = 1}^k\prod_{j = 1}^l
		\mathds{1}_{(-\infty,r_i)}(\langle\varphi_i,\mu\rangle)
		\mathds{1}_{(-\infty,s_j)}( \langle\phi_j,\nu\rangle)
		\Gamma_1(d\mu,d\nu)
		\\= 
		\int_{\Mc^2_K(\Cc^d)}
		\prod_{i = 1}^k\prod_{j = 1}^l
		\mathds{1}_{(-\infty,r_i)}(\langle\varphi_i,\mu\rangle)
		\mathds{1}_{(-\infty,s_j)}( \langle\phi_j,\nu\rangle)
		\Gamma_2(d\mu,d\nu).
	\end{align*}
	Consequently, we can say that
	$\Gamma_1 = \Gamma_2$ on $B$, where
	\begin{align*}
		B := \Big\{\{(\mu,\nu) \in \Mc^2_K(\Cc^d): \langle\varphi_i,\mu\rangle < r_i, &
		\langle\phi_j,\nu\rangle < s_j, i \in [k], j \in [l]\}
		\\&: k,l \ge 1,r_i,s_j \in \R, \varphi_i, \phi_j \in C_b(\Cc^d),  i \in [k], j \in [l]\Big\}.
	\end{align*}
	Since the weak convergence topology on $\Mc_{2,K}(\Cc^d)$ can be characterized by bounded continuous functions, the Borel $\sigma$-field on  $\Mc_{2,K}(\Cc^d)$ can be generated by all open sets in $\Mc_{2,K}(\Cc^d)$, and the product of Borel $\sigma$-fields on $\Mc_{2,K}(\Cc^d)$ is equal to Borel $\sigma$-field on the product space $\Mc^2_{2,K}(\Cc^d)$, it follows that $\Bc(\Mc^2_{2,K}(\Cc^d)) = \sigma(B)$. Then, we can conclude the proof by a monotone class argument.
\end{proof}

Similarly, we have the following result.
\begin{thm}\label{thm: characterization2}
	Suppose that $\Gamma_1, \Gamma_2 \in \Pc(\Pc(\Cc^d) \x \Pc(\Cc^d))$ satisfy
	\begin{align}
		\int_{\Pc^2(\Cc^d)} \prod_{i = 1}^k\prod_{j = 1}^l \langle\varphi_i,\mu\rangle \langle\phi_j,\nu\rangle \Gamma_1(d\mu,d\nu)
		= 
		\int_{\Pc^2(\Cc^d)} \prod_{i = 1}^k\prod_{j = 1}^l \langle\varphi_i,\mu\rangle \langle\phi_j,\nu\rangle \Gamma_2(d\mu,d\nu)
	\end{align}
	for all $\varphi_i, \phi_j \in C_b(\Cc^d)$, $i \in [k]$, $j \in [l]$, $k, l \ge 1$. 
	Then, $\Gamma_1 = \Gamma_2$.
\end{thm}

\bigskip

\section{Linear graphon system}   \label{sec: linear case}
In the graphon particle system \eqref{eq : graphon particle system with common noise} and the finite $n$-particle system \eqref{eq : finite particle system with common noise}, suppose that the functions $\phi_{\alpha}$ take the specific form
\begin{equation*}
	\phi_{\alpha}(x, \mu) := \int_{\R^d} \hat{\phi}_{\alpha}(x, y) d\mu(y)
\end{equation*}
for some bounded Lipschitz function $\hat{\phi}_{\alpha}$, defined on $\R^d \times \R^d$, for $\alpha = p, b, w$:
\begin{align*}
	|\hat{\phi}_{\alpha}(x_1, y_1) | \le L, \quad |\hat{\phi}_{\alpha}(x_1, y_1) - \hat{\phi}_{\alpha}(x_2, y_2)| \le L\big( |x_1-x_2|+|y_1-y_2| \big), \quad \forall \, x_1, x_2, y_1, y_2 \in \R^d.
\end{align*}
In this case, interactions between the agents are linear in their empirical measures, and the two systems take similar linear forms as in \cite{BCW, Bayraktar_Kim_2024, BW2, BW, BWZ}.

To show that all the results in this paper apply to the linear graphon systems, we claim that this specific form of $\phi_{\alpha}$ satisfies Assumption \ref{ass:Lipschitz} by establishing for any $x, y \in \R^d$, $\mu, \nu \in \mathcal{M}_{2,K}(\R^d)$, and $\alpha = p, b, w$
\begin{equation}    \label{angle Lipschitz}
	\big| \phi_{\alpha}(x, \mu) - \phi_{\alpha}(y, \mu) \big| \le L m_{\mu}|x-y|, \quad 
	\big| \phi_{\alpha}(x, \mu) - \phi_{\alpha}(x, \nu) \big|^2 \le L^2 \Wop^2_2(\mu, \nu).
\end{equation}
The first inequality of \eqref{angle Lipschitz} is straightforward:
\begin{align*}
	\big| \phi_{\alpha}(x, \mu) - \phi_{\alpha}(y, \mu) \big|
	\le \int_{\R^d} \Big| \hat{\phi}_{\alpha}(x, z) - \hat{\phi}_{\alpha}(y, z) \Big| d\mu(z)
	\le \int_{\R^d} L |x-y| d\mu(z) = L m_{\mu} |x-y|.
\end{align*}

For the second inequality, consider the two probability measures $\mu^\circ$, $\nu^\circ$, recalling the notation from Section \ref{subsec: Wop}. Let us first assume $m_{\mu} m_{\nu} \neq 0$. Consider an arbitrary probability measure $\bar{\gamma}$ on $\R^d \times \R^d$ with marginals $T_{m_{\mu}\#}\bar{\mu}$ and $T_{m_{\nu}\#}\bar{\nu}$, and use the properties of $\hat{\phi}_{\alpha}$ to derive
\begin{align}
	&\big| \phi_{\alpha}(x, \mu) - \phi_{\alpha}(x, \nu) \big|^2
	= \bigg| \int_{\R^d} m_{\mu} \hat{\phi}_{\alpha}(x, y) d\bar{\mu}(y) - \int_{\R^d} m_{\nu} \hat{\phi}_{\alpha}(x, z) d\bar{\nu}(z) \bigg|^2             \nonumber
	\\
	& \qquad \le 2\bigg| \int_{\R^d} m_{\mu} \hat{\phi}_{\alpha}(x, y) d\bar{\mu}(y) - \int_{\R^d} m_{\mu} \hat{\phi}_{\alpha}(x, z) d\bar{\nu}(z) \bigg|^2             \label{zero mass}
	\\
	& \qquad \qquad \qquad \qquad \qquad \qquad + 2\bigg| \int_{\R^d} m_{\mu} \hat{\phi}_{\alpha}(x, z) d\bar{\nu}(z) - \int_{\R^d} m_{\nu} \hat{\phi}_{\alpha}(x, z) d\bar{\nu}(z) \bigg|^2             \nonumber
	\\
	& \qquad \le 2\bigg| \int_{\R^d} m_{\mu} \hat{\phi}_{\alpha}(x, \frac{y}{m_{\mu}}) dT_{m_{\mu}\#}\bar{\mu}(y) - \int_{\R^d} m_{\mu} \hat{\phi}_{\alpha}(x, \frac{z}{m_{\nu}}) dT_{m_{\nu}\#}\bar{\nu}(z) \bigg|^2 + 2L^2 (m_{\mu}-m_{\nu})^2             \nonumber
	\\
	& \qquad \le 2 L^2 m_{\mu}^2 \iint_{\R^d \times \R^d} \Big| \frac{y}{m_{\mu}} - \frac{z}{m_{\nu}} \Big|^2 d\bar{\gamma}(y, z) + 2L^2 (m_{\mu}-m_{\nu})^2             \nonumber
	\\
	& \qquad \le 4 L^2 m_{\mu}^2 \iint_{\R^d \times \R^d} \Big| \frac{y}{m_{\mu}} - \frac{z}{m_{\mu}} \Big|^2 d\bar{\gamma}(y, z)               \nonumber
	\\
	& \qquad \qquad \qquad \qquad \qquad \qquad + 4 L^2 m_{\mu}^2 \iint_{\R^d \times \R^d} \Big| \frac{z}{m_{\mu}} - \frac{z}{m_{\nu}} \Big|^2 d\bar{\gamma}(y, z) + 2L^2 (m_{\mu}-m_{\nu})^2             \nonumber
	\\
	& \qquad \le 4L^2 \iint_{\R^d \times \R^d} |y-z|^2 d\bar{\gamma}(y, z) + 4L^2 m_{\mu}^2 \Big(\frac{1}{m_{\mu}}-\frac{1}{m_{\nu}}\Big)^2 M_{0}(T_{m_{\nu}\#}\bar{\nu}) + 2L^2 (m_{\mu}-m_{\nu})^2.             \nonumber
\end{align}
By taking the minimum over such probability measures $\bar{\gamma}$ with marginals $T_{m_{\mu}\#}\bar{\mu}$ and $T_{m_{\nu}\#}\bar{\nu}$, we obtain
\begin{align*}
	\big| \phi_{\alpha}(x, \mu) - \phi_{\alpha}(x, \nu) \big|^2
	&\le 4 L^2 \Wc_2^2(T_{m_{\mu}\#}\bar{\mu}, T_{m_{\nu}\#}\bar{\nu})
	\\
	& \qquad \qquad \qquad+ 4L^2 \frac{(m_{\nu}-m_{\mu})^2}{m_{\nu}^2} M_{0}(T_{m_{\nu}\#}\bar{\nu}) + 2L^2 (m_{\mu}-m_{\nu})^2
	\\
	&\le 4 L^2 \Wc_2^2(T_{m_{\mu}\#}\bar{\mu}, T_{m_{\nu}\#}\bar{\nu}) + (m_{\nu}-m_{\mu})^2 \Big( \frac{4L^2}{m_{\nu}^2} M_{0}(T_{m_{\nu}\#}\bar{\nu}) + 2L^2 \Big)
	\\
	&\le \max\Big\{4L^2, ~\frac{4L^2}{m_{\nu}^2} M_{0}(T_{m_{\nu}\#}\bar{\nu}) + 2L^2 \Big\} \Wop_2^2 (\mu, \nu)
	\\
	& = \max\Big\{4L^2, ~4L^2 M_{0}(\bar{\nu}) + 2L^2 \Big\} \Wop_2^2 (\mu, \nu)
	\\
	& \le \max\Big\{4L^2, ~4L^2 K + 2L^2 \Big\} \Wop_2^2 (\mu, \nu).
\end{align*}
Here, the third inequality uses the definition \eqref{def: WOP}, and this proves the second Lipschitz property of \eqref{angle Lipschitz}.

When one of $\{m_{\mu}, m_{\nu}\}$ is zero, say $m_{\mu} = 0$, the first term of \eqref{zero mass} is zero, so we have $| \phi_{\alpha}(x, \mu) - \phi_{\alpha}(x, \nu)|^2 \le 2L^2(m_{\mu}-m_{\nu})^2$  and the result \eqref{angle Lipschitz} holds.
\end{appendix}
\bigskip

\bigskip
\noindent \textbf{Funding}

\smallskip

\noindent E. Bayraktar is supported in part by the National Science Foundation and by the Susan M. Smith Professorship.

\noindent D. Kim is supported by the National Research Foundation of Korea (NRF) grant funded by the Korea government (MSIT) RS-2025-00513609 and RS-2019-NR040050.

\bibliographystyle{plain}   
\bibliography{ref}

@INPROCEEDINGS{Graphon:Qnoise,
  author={Dunyak, Alex and Caines, Peter E.},
  booktitle={2022 IEEE 61st Conference on Decision and Control (CDC)}, 
  title={Linear Stochastic Graphon Systems with Q-Space Noise}, 
  year={2022},
  volume={},
  number={},
  pages={3926-3932},
  doi={10.1109/CDC51059.2022.9992862}
}

@article{Xu:Gou,
author = {Xu, De-xuan and Gou, Zhun and Huang, Nan-jing and Gao, Shuang},
title = {Linear-Quadratic Graphon Mean Field Games with Common Noise},
journal = {SIAM Journal on Control and Optimization},
volume = {63},
number = {5},
pages = {3526-3556},
year = {2025},
doi = {10.1137/24M1637167},
URL = {https://doi.org/10.1137/24M1637167},
eprint = {https://doi.org/10.1137/24M1637167}
}

@article{Tanpi:Zhou,
author = {Ludovic Tangpi and Xuchen Zhou},
title = {{Optimal investment in a large population of competitive and heterogeneous agents}},
volume = {28},
journal = {Finance and Stochastics},
publisher = {Springer},
pages = {497 -- 551},
year = {2024},
doi = {10.1007/s00780-023-00527-9},
URL = {https://doi.org/10.1007/s00780-023-00527-9}
}

@misc{chen2025graphonparticlesystemsii,
      title={{Graphon Particle Systems, Part II: Dynamics of Distributed Stochastic Continuum Optimization}}, 
      author={Yan Chen and Tao Li},
      year={2025},
      note={arXiv:2407.02765},
      url={https://arxiv.org/abs/2407.02765}, 
}

@article{sun2009individual,
  title={Individual risk and Lebesgue extension without aggregate uncertainty},
  author={Sun, Yeneng and Zhang, Yongchao},
  journal={Journal of Economic Theory},
  volume={144},
  number={1},
  pages={432--443},
  year={2009},
  publisher={Elsevier}
}

@article{Coppini:Graphon,
	TITLE = {A note on {F}okker-{P}lanck equations and graphons},
	AUTHOR = {Fabio Coppini},
	YEAR = {2022},
	journal = {Journal of Statistical Physics},
	volume = {187},
	doi = {10.1007/s10955-022-02905-7}
}

@article{bayraktar2024nonparametricestimatesgraphonmeanfield,
      title={Non-parametric estimates for graphon mean-field particle systems}, 
      author={Erhan Bayraktar and Hongyi Zhou},
journal = {To appear in Bernoulli. Available on ArXiv.},	      
year={2025+},
}

@article{BW2,
	title = {Stationarity and uniform in time convergence for the graphon particle system},
	journal = {Stochastic Processes and their Applications},
	volume = {150},
	pages = {532-568},
	year = {2022},
	issn = {0304-4149},
	doi = {https://doi.org/10.1016/j.spa.2022.04.006},
	url = {https://www.sciencedirect.com/science/article/pii/S0304414922000862},
	author = {Erhan Bayraktar and Ruoyu Wu}
}

@article{Bayraktar_Kim_2024,
    title={Concentration of measure for graphon particle system},
    volume={56},
    DOI={10.1017/apr.2023.59},
    number={4},
    journal={Advances in Applied Probability},
    author={Bayraktar, Erhan and Kim, Donghan},
    year={2024},
    pages={1279–1306}
}

@article{BCW,
author = {Erhan Bayraktar and Suman Chakraborty and Ruoyu Wu},
title = {{Graphon mean field systems}},
volume = {33},
journal = {The Annals of Applied Probability},
number = {5},
publisher = {Institute of Mathematical Statistics},
pages = {3587 -- 3619},
year = {2023},
doi = {10.1214/22-AAP1901},
URL = {https://doi.org/10.1214/22-AAP1901}
}

@article{BW,
title = {Graphon particle system: Uniform-in-time concentration bounds},
journal = {Stochastic Processes and their Applications},
volume = {156},
pages = {196--225},
year = {2023},
issn = {0304-4149},
doi = {https://doi.org/10.1016/j.spa.2022.11.008},
url = {https://www.sciencedirect.com/science/article/pii/S030441492200240X},
author = {Erhan Bayraktar and Ruoyu Wu}
}

@article{BWZ,
	journal = {Applied Mathematics \& Optimization},
	author = {Bayraktar, Erhan and Wu, Ruoyu and Zhang, Xin},
	title = {{Propagation of Chaos of Forward-Backward Stochastic Differential Equations with Graphon Interactions}},
	volume={88},
	issue = {1},
	year = {2023}
}

@article{Gao:Tchuendom:Caines,
	title={Linear Quadratic Graphon Field Games},
	author={Shuang Gao and Rinel Foguen Tchuendom and Peter E. Caines},
	journal={Communications in Information and Systems},
	year={2021},
	volume={21},
	pages={341-369}
}

@INPROCEEDINGS{GCH,
	author={Gao, Shuang and Caines, Peter E. and Huang, Minyi},
	booktitle={2021 60th IEEE Conference on Decision and Control (CDC)}, 
	title={L{Q}{G} Graphon Mean Field Games: Graphon Invariant Subspaces}, 
	year={2021},
	pages={5253-5260},
	doi={10.1109/CDC45484.2021.9683037}
}

@INPROCEEDINGS{VMV,
	author={Vasal, Deepanshu and Mishra, Rajesh and Vishwanath, Sriram},
	booktitle={2021 American Control Conference (ACC)}, 
	title={Sequential Decomposition of Graphon Mean Field Games}, 
	year={2021},
	volume={},
	number={},
	pages={730-736},
	doi={10.23919/ACC50511.2021.9483331}
}

@INPROCEEDINGS{TCH,
	author={Tchuendom, Rinel Foguen and Caines, Peter E. and Huang, Minyi},
	booktitle={2021 60th IEEE Conference on Decision and Control (CDC)}, 
	title={Critical Nodes in Graphon Mean Field Games}, 
	year={2021},
	pages={166-170},
	doi={10.1109/CDC45484.2021.9683629}
}

@INPROCEEDINGS{TCH2,
	author={Tchuendom, Rinel Foguen and Caines, Peter E. and Huang, Minyi},
	booktitle={2020 59th IEEE Conference on Decision and Control (CDC)}, 
	title={On the Master Equation for Linear Quadratic Graphon Mean Field Games}, 
	year={2020},
	volume={},
	number={},
	pages={1026-1031},
	doi={10.1109/CDC42340.2020.9304291}
}

@INPROCEEDINGS{graphon:game:GMFG,
  author={Caines, Peter E. and Huang, Minyi},
  booktitle={2018 IEEE Conference on Decision and Control (CDC)}, 
  title={{Graphon Mean Field Games and the GMFG Equations}}, 
  year={2018},
  volume={},
  number={},
  pages={4129-4134},
  doi={10.1109/CDC.2018.8619367}
}

@article{Lacker:graphon_games,
author = {Lacker, Daniel and Soret, Agathe},
title = {A Label-State Formulation of Stochastic Graphon Games and Approximate Equilibria on Large Networks},
journal = {Mathematics of Operations Research},
volume = {48},
number = {4},
pages = {1987-2018},
year = {2023},
doi = {10.1287/moor.2022.1329},
URL = {https://doi.org/10.1287/moor.2022.1329},
eprint = {https://doi.org/10.1287/moor.2022.1329}
}

@article{graphon:epidem,
title = {Local-density dependent Markov processes on graphons with epidemiological applications},
journal = {Stochastic Processes and their Applications},
volume = {148},
pages = {324-352},
year = {2022},
issn = {0304-4149},
doi = {https://doi.org/10.1016/j.spa.2022.03.001},
url = {https://www.sciencedirect.com/science/article/pii/S0304414922000576},
author = {{Dániel Keliger and Illés Horváth and Bálint Takács}}
}

@article{graphon:econometrica,
author = {Parise, Francesca and Ozdaglar, Asuman},
title = {Graphon Games: A Statistical Framework for Network Games and Interventions},
journal = {Econometrica},
volume = {91},
number = {1},
pages = {191-225},
doi = {https://doi.org/10.3982/ECTA17564},
url = {https://onlinelibrary.wiley.com/doi/abs/10.3982/ECTA17564},
eprint = {https://onlinelibrary.wiley.com/doi/pdf/10.3982/ECTA17564},
year = {2023}
}

@article{toymodel:Erdos-Renyi,
	author = {{Delarue, François}},
	title = {{Mean field games: A toy model on an Erd\"{o}s-Renyi graph.}},
	DOI= "10.1051/proc/201760001",
	url= "https://doi.org/10.1051/proc/201760001",
	journal = {ESAIM: Procs},
	year = 2017,
	volume = 60,
	pages = "1-26",
}

@article{Graphon:game,
author = {Caines, Peter E. and Huang, Minyi},
title = {Graphon Mean Field Games and Their Equations},
journal = {SIAM Journal on Control and Optimization},
volume = {59},
number = {6},
pages = {4373-4399},
year = {2021},
doi = {10.1137/20M136373X},
URL = {https://doi.org/10.1137/20M136373X},
eprint = {https://doi.org/10.1137/20M136373X}
}

@article{Carmona:graphon_games,
author = {Carmona, Ren\'{e} and Cooney, Daniel B. and Graves, Christy V. and Lauri\`{e}re, Mathieu},
title = {Stochastic Graphon Games: I. The Static Case},
journal = {Mathematics of Operations Research},
volume = {47},
number = {1},
pages = {750-778},
year = {2022},
doi = {10.1287/moor.2021.1148},
URL = {https://doi.org/10.1287/moor.2021.1148},
eprint = {https://doi.org/10.1287/moor.2021.1148}
}

@article{Aurell2022,
  author       = {Alexander Aurell and Ren{\'e} Carmona and Mathieu Lauri{\`e}re},
  title        = {{Stochastic Graphon Games: II. The Linear-Quadratic Case}},
  journal      = {Applied Mathematics \& Optimization},
  year         = {2022},
  volume       = {85},
  number       = {3},
  pages        = {39},
  doi          = {10.1007/s00245-022-09839-2},
  url          = {https://doi.org/10.1007/s00245-022-09839-2},
  issn         = {1432-0606},
  date         = {2022-05-10}
}

@article{SUN200631,
title = {{The exact law of large numbers via Fubini extension and characterization of insurable risks}},
journal = {Journal of Economic Theory},
volume = {126},
number = {1},
pages = {31-69},
year = {2006},
issn = {0022-0531},
doi = {https://doi.org/10.1016/j.jet.2004.10.005},
url = {https://www.sciencedirect.com/science/article/pii/S0022053104002224},
author = {Yeneng Sun}
}

@article{nonlinear:graphon,
      title={Nonlinear Graphon mean-field systems}, 
      author={Fabio Coppini and Anna De Crescenzo and Huyen Pham},
journal = {To appear in Stochastic Processes and Their Applications, available on ArXiv.},      
year={2025+}, 
}

@book{villani,
  title={Topics in optimal transportation},
  author={Villani, C{\'e}dric},
  volume={58},
  year={2021},
  publisher={American Mathematical Soc.}
}

@book{villani2016optimal,
  title={Optimal Transport: Old and New},
  author={Villani, C.},
  isbn={9783662501801},
  series={Grundlehren der mathematischen Wissenschaften},
  url={https://books.google.co.kr/books?id=5p8SDAEACAAJ},
  year={2016},
  publisher={Springer Berlin Heidelberg}
}

@article{sznitman,
  title={Topics in propagation of chaos},
  author={Sznitman, Alain-Sol},
  journal={Ecole d’{\'e}t{\'e} de probabilit{\'e}s de Saint-Flour XIX—1989},
  volume={1464},
  pages={165--251},
  year={1991},
  publisher={Springer}
}

@book{billingsley,
  title={Convergence of probability measures},
  author={Billingsley, Patrick},
  year={2013},
  publisher={John Wiley \& Sons}
}

@unpublished{extendingWmetric,
      title={{Extending the Wasserstein metric to positive measures}}, 
      author={Hugo Leblanc and Thibaut Le Gouic and Jacques Liandrat and Magali Tournus},
      year={2023},
      eprint={2303.02183},
      note ={arXiv:2303.02183},
      primaryClass={math.MG},
      url={https://arxiv.org/abs/2303.02183}, 
}

@article{DTP,
  title={{McKean--Vlasov optimal control: limit theory and equivalence between different formulations}},
  author={Djete, Mao Fabrice and Possama{\"\i}, Dylan and Tan, Xiaolu},
  journal={Mathematics of Operations Research},
  volume={47},
  number={4},
  pages={2891--2930},
  year={2022},
  publisher={INFORMS}
}

@article{FournierGuillin,
  title={{On the rate of convergence in Wasserstein distance of the empirical measure}},
  author={Fournier, Nicolas and Guillin, Arnaud},
  journal={Probability theory and related fields},
  volume={162},
  number={3},
  pages={707--738},
  year={2015},
  publisher={Springer}
}

@article{lei,
author = {Jing Lei},
title = {{Convergence and concentration of empirical measures under Wasserstein distance in unbounded functional spaces}},
volume = {26},
journal = {Bernoulli},
number = {1},
publisher = {Bernoulli Society for Mathematical Statistics and Probability},
pages = {767 -- 798},
year = {2020},
doi = {10.3150/19-BEJ1151},
URL = {https://doi.org/10.3150/19-BEJ1151}
}

@book{lovasz,
  title={Large networks and graph limits},
  author={Lov{\'a}sz, L{\'a}szl{\'o}},
  volume={60},
  year={2012},
  publisher={American Mathematical Soc.}
}

\end{document}